\documentclass[a4aper,10pt,reqno]{amsart}

\usepackage[english]{babel}
\usepackage[utf8]{inputenc}
\usepackage{xcolor}
\usepackage{amsmath,amssymb,amsfonts,amsthm,amscd}
\usepackage[mathscr]{eucal}
\usepackage{mathtools}
\usepackage{enumitem}
\usepackage{hyperref}
\usepackage{booktabs,subcaption,dcolumn,multirow}
\usepackage[normalem]{ulem}
\usepackage[a4paper]{geometry}
\newcolumntype{d}[1]{D..{#1}}
\newtheorem{theorem}{Theorem}[section]
\newtheorem{proposition}[theorem]{Proposition}
\newtheorem{corollary}[theorem]{Corollary}
\newtheorem{lemma}[theorem]{Lemma}

\theoremstyle{definition}
\newtheorem{definition}{Definition}[section]

\theoremstyle{remark}
\newtheorem{remark}[theorem]{Remark}

\numberwithin{equation}{section}

\title[Uniqueness of the Bonnet problem in Thurston geometries]{Uniqueness of the Bonnet problem in Thurston geometries}
\author{Jos\'e S. Santiago}
\address{Departamento de Matem\'aticas, Universidad de Ja\'en, 23071 Ja\'en SPAIN}
\email{jsantivillanueva@gmail.com}



\begin{document}
	\maketitle
\begin{abstract}
We study the Bonnet problem in Bianchi–Cartan–Vr\u{a}nceanu spaces and in $\mathrm{Sol}_3$. Our main contribution is to establish the uniqueness of Bonnet mates, which leads us to address the problem of determining when an isometric immersion can be continuously deformed through isometric immersions that preserve the principal curvatures —a question originally posed in $\mathbb{R}^3$ by Chern~\cite{Chern}. For Bianchi–Cartan–Vr\u{a}nceanu spaces, we complete the local classification of Bonnet pairs by studying the uniqueness of the results obtained by Gálvez, Martínez and Mira~\cite{GMM}, and we provide new examples of Bonnet mates that were not previously considered. In particular, we prove that the aforesaid continuous deformations only exist for minimal surfaces in the product spaces $\mathbb{S}^2\times\mathbb{R}$ and $\mathbb{H}^2\times\mathbb{R}$ and otherwise only for surfaces with constant principal curvatures. In the case of $\mathrm{Sol}_3$, we give a characterization of Bonnet mates via a system of two differential equations, addressing a problem proposed in~\cite{GMM}. We conclude that the only surfaces admitting continuous isometric deformations that preserve the principal curvatures in $\mathrm{Sol}_3$ are those with constant left-invariant Gauss map.
\end{abstract}
\section{Introduction}
The \emph{Bonnet problem} is a classical question in surface theory which, given a Riemannian surface $(\Sigma,  s^2)$ isometrically immersed in $\mathbb{R}^3$, seeks to classify all isometric immersions of $\Sigma$ into $\mathbb{R}^3$ that share the same mean curvature. Originally, Bonnet proved that every surface in $\mathbb{R}^3$ with constant mean curvature admits an $\mathbb{S}^1$-family of isometric immersions with the same mean curvature. Later, Heléin~\cite{Helein01} reformulated this deformation in terms of \emph{moving frames}. Note that, in this context, two isometric immersions share the mean curvature if and only if they share the principal curvatures, since two Bonnet mates necessarily share the determinant of their shape operator due to the Gauss equation $K = \det(S) = \kappa_1\kappa_2$, where $\kappa_1$, $\kappa_2$ are the principal curvatures of the immersion and $S$ its shape operator. Taking this into account, Chern~\cite{Chern} generalized Bonnet's examples by studying continuous isometric deformations of surfaces in $\mathbb{R}^3$ that preserve the principal curvatures, concluding that the only surfaces admitting such deformations are those with constant mean curvature or those satisfying a certain first-order differential equation whose solutions depend on $6$ arbitrary constants.

In $\mathbb{R}^3$, the Bonnet problem has been studied in the simply connected case from several perspectives: via quaternions, by Kamberov, Pedit and Pinkall~\cite{KamberovPeditPinkall1998}; through integrable systems, by Bobenko and Eitner~\cite{BobenkoEitner1998}; and using moving frames, by Jensen, Musso and Nicolodi~\cite{JensenMussoNicolodi}. The key concept in the resolution of this problem is that of an isothermic surface, i.e., a surface that admits a local conformal parametrization which also diagonalizes the shape operator. In this way, it is shown that any simply connected surface without umbilical points and non-constant mean curvature admits exactly one Bonnet mate if it is not isothermic, and admits at least two if it is isothermic. Moreover, Chen and Li~\cite{ChenHaizhong} extended these results in an analogous way to the spaces $\mathbb{S}^3$ and $\mathbb{H}^3$ (see also the work of Bobenko and Eitner~\cite{BobenkoEitner00}). In these works, the surface is usually assumed to be simply connected, and the general case of a global version remains open. In this sense, Bobenko, Hoffman and Sageman~\cite{BobenkoHoffmanSageman} recently constructed the first example of a torus admitting a Bonnet mate in $\mathbb{R}^3$.

Parallel to these developments, growing attention has been given to the so-called Thurston geometries: $\mathbb{R}^3$, $\mathbb{S}^3$, $\mathbb{H}^3$, $\mathbb{S}^2\times\mathbb{R}$, $\mathbb{H}^2\times\mathbb{R}$, $\mathrm{Nil}_3$, $\widetilde{\mathrm{SL}}_2(\mathbb{R})$ and $\mathrm{Sol}_3$, homogeneous $3$-manifolds equipped with metrics for which the isometry group is maximal (see \cite{Scott83}). Since the Bonnet problem is already fully understood in the simply connected space form case, our attention turns to the remaining geometries. We work in the Bianchi–Cartan–Vrănceanu spaces (also denoted by $\mathbb{E}(\kappa,\tau)$), that is, the simply connected, complete homogeneous $3$–manifolds whose isometry group has dimension four. These also include the Berger spheres $\mathbb{S}^3_b$, which are not part of Thurston’s list. Together with these spaces, the only geometry left to consider is $\mathrm{Sol}_3$, whose isometry
group has dimension $3$, and which we also treat.

A key difficulty in extending the Bonnet problem to homogeneous $3$–manifolds is that, in general, the Gauss equation involves several angle functions $\nu_1,\nu_2,\nu_3$ (see~\cite{CMS}). However, a feature shared by the geometries $\mathbb{E}(\kappa,\tau)$ and $\mathrm{Sol}_3$ is that their Gauss equation depends only on a single angle function~$\nu_3$:
\[
\begin{array}{ll}
	K = \kappa_1 \kappa_2 + \tau^2 + (\kappa - 4\tau^2)\nu_3^2, & \hspace{1em} (\textit{in $\mathbb{E}(\kappa,\tau)$}), \\[6pt]
	K = \kappa_1 \kappa_2 - \mu^2 + 2\mu^2 \nu_3^2,               & \hspace{1em} (\textit{in $\mathrm{Sol}_3$}).
\end{array}
\]
While in space forms the Bonnet problem is interpreted as the classification of isometric immersions that share the mean curvature, in general it is described as the classification of those that share the principal curvatures. If two isometric immersions of the same Riemannian surface share the principal curvatures and are not congruent in the ambient space, they are said to be \emph{Bonnet mates}. A natural class of immersions that tend to admit Bonnet mates are those with some kind of intrinsic symmetry that does not arise from an ambient symmetry (see Section~\ref{subsec:invariantes}).

Gálvez, Martínez, and Mira~\cite{GMM} studied the Bonnet problem in the spaces 
$\mathbb{E}(\kappa,\tau)$ using complex analytic techniques. They proved that when 
$\tau = 0$, that is, in the product spaces $\mathbb{M}^2(\kappa)\times\mathbb{R}$, 
an immersion admits a Bonnet mate if and only if it is minimal or properly 
invariant under a one-parameter group of isometries, and they provided explicit 
examples of both types. In the minimal case, these examples are the associated family obtained independently by Daniel~\cite{Dan09} and by Hauswirth, Sa~Earp and Toubiana~\cite{HST}. When $\tau \neq 0$, they showed that an immersion admits a Bonnet mate if and only if it has constant mean curvature or is invariant under a one-parameter group of isometries, again providing explicit examples. In this setting the problem becomes more subtle, since orientation plays a fundamental role and leads to two distinct types of Bonnet mates: \emph{positive} Bonnet mates, whose immersions induce the same orientation on the surface, and \emph{negative} Bonnet mates, which induce opposite orientations. The results in \cite{GMM} do not address uniqueness, since their examples need not be unique — and in fact, in some cases they are not. This type of result is far from obvious, as is clear from Daniel’s proof of the uniqueness of Bonnet mates for minimal surfaces in product spaces in~\cite{Daniel15}.

In the product case $\tau = 0$ (see Theorem~\ref{th:main theorem tau 0}), we show that an isometric immersion admits exactly a continuous $\mathbb{S}^1$-family of Bonnet mates if and only if the immersion is minimal or has constant principal curvatures. Otherwise, the Bonnet mate is unique. In the case $\tau \neq 0$ (see Theorems~\ref{teo1} and~\ref{teo2}), we show that an isometric immersion admits an $\mathbb{S}^1$-family of positive Bonnet mates, and another $\mathbb{S}^1$-family of negative Bonnet mates, precisely when it has constant principal curvatures. The positive Bonnet mate is unique in all other cases; for negative Bonnet mates, the immersion admits exactly two of them when it is simultaneously invariant under a one-parameter group of isometries and has nonzero constant mean curvature, and otherwise the negative mate is unique.

This completes the classification initiated in~\cite{GMM} and solves Chern’s problem in $\mathbb{E}(\kappa,\tau)$ (see Corollaries~\ref{cortau0} and~\ref{cortauneq0}): the only immersions that admit isometric deformations preserving the principal curvatures are the minimal ones in the product case $\tau = 0$ and, for $\tau \neq 0$, those with constant principal curvatures. By the work of Domínguez\mbox{-}Vázquez and Manzano~\cite{DM}, surfaces with constant principal curvatures coincide with the extrinsically homogeneous ones (these have been studied more generally by Domínguez-Vázquez, Ferreira and Otero~\cite{DFT25} in Lie groups), which are $2$-dimensional Lie subgroups. It is interesting that these surfaces appear in our arguments, since (by \cite{DM}) they coincide with the immersions with constant $\nu_3$ and $H$ studied by Espinar and Rosenberg~\cite{EspinarRosenberg11}.

In the case of $\mathrm{Sol}_3$, the study of surfaces has gained considerable attention thanks to several important works, such as those of Daniel and Mira~\cite{DanielMira13}, Meeks~\cite{Meeks13}, and Inoguchi and Lee~\cite{InoguchiLee08}. Regarding the Bonnet problem, however, no results were previously known in this setting; in fact, it is an open problem proposed in~\cite{GMM}. We prove that every immersion properly invariant (see Definition~\ref{def}) under a one-parameter group of isometries admits a Bonnet mate (see Proposition~\ref{propSol3inv}). Moreover, we characterize when a general immersion admits a Bonnet mate, in analogy with Chern~\cite{Chern}, by means of two first-order differential equations (see Theorem~\ref{pr:characterization Sol Bonnet mates}). By studying this system of equations in greater detail, we show that, except in the case of isometric immersions whose left-invariant Gauss map is constant, there exist at most $7$ Bonnet mates (see Theorem~\ref{th:last theorem}). As a consequence, we solve Chern’s problem (see Corollary~\ref{corChernSol3}) by showing that the only immersions admitting continuous isometric deformations that preserve the principal curvatures are those with constant left-invariant Gauss map.

\medskip

\noindent\textbf{Acknowledgement.} This work was supported by the research grant PID2022-142559NB-I00, which is funded by MCIN/AEI/10.13039/501100011033, and constitutes part of the author's PhD thesis. The author is grateful to José M. Manzano Prego for many helpful comments that improved this text.

\section{Preliminaries on Thurston geometries}
Thurston geometries are the spaces $\mathbb{R}^3$, $\mathbb{S}^3$, $\mathbb{H}^3$, $\mathbb{S}^2\times \mathbb{R}$, $\mathbb{H}^2\times \mathbb{R}$, $\widetilde{\mathrm{SL}}_2(\mathbb{R})$, $\mathrm{Nil}_3$ and $\mathrm{Sol}_3$, endowed with homogeneous metrics whose isometry groups are maximal. We focus on the cases that are not space forms, and we work in a slightly more general setting, since we also consider the Berger spheres, namely $\mathrm{SU}(2)$ with left-invariant metrics whose isometry group has dimension $4$. More explicitly, we will consider: the Bianchi--Cartan--Vr\u{a}nceanu spaces, which are the simply connected, complete, homogeneous Riemannian $3$-manifolds with a $4$-dimensional isometry group (see Subsection~\ref{subsec:description Ekt spaces}); the unimodular solvable Lie group $\mathrm{Sol}_3$ equipped with a special left-invariant Riemannian metric whose isotropy group at each point has $8$ elements (see Subsection~\ref{subsec:description Sol3 type spaces}). We give a brief review of some fundamental aspects of the structure and geometry of these spaces, as well as the theory of isometric immersions in these spaces. Throughout, we denote any of these three-dimensional oriented manifolds by $M$.

\begin{remark}
	We work in the setting of real-analytic manifolds and real-analytic immersions (see~\cite[Ch.1, Sec. 1.2]{Varadarajan84}). The main reason for adopting this framework is that it allows us to use an identity principle: if two real-analytic immersions agree on a nonempty open set, then they agree on their entire maximal common domain. We will make use of this principle, although our arguments remain valid under the weaker assumption that $f$ is of class $C^3$. This change does not affect the proofs, but only the precise statements, as already noted in~\cite{GMM}. \hfill $\blacksquare$
\end{remark}

Let $(\Sigma, \mathrm{d} s^2)$ be a real-analytic Riemannian surface and let $f:\Sigma \rightarrow M$ be a real-analytic isometric immersion. All geometric objects will be assumed to have real-analytic regularity. We identify $T_p \Sigma \cong \mathrm{d} f_p (T_p\Sigma) \subset T_{f(p)}M$ for any $p \in \Sigma$ and omit the reference to the point when it is clear from the context. Let $N$ be a unit normal vector field along $f$. The orientation of the ambient space, together with $N$, induces an orientation on $\Sigma$ via a rotation by $\tfrac{\pi}{2}$, which we denote by $J$. This rotation is described by an endomorphism field $J$ on the tangent bundle of $\Sigma$ such that $J^2 = -\mathrm{Id}$. It is defined by requiring that the frame $\{u, Ju, N\}$ is positively oriented, or equivalently, $Ju = N \times u$ for every nonzero vector $u \in T_p\Sigma$ and every $p \in \Sigma$.

The normal vector field $N$ also determines a symmetric $(1,1)$-tensor on $\Sigma$ which can be identified with the adjoint operator known as the \emph{shape operator}:
\[
S_p:T_p\Sigma \to T_p\Sigma, \qquad S_p(u) = -\overline{\nabla}_{u} N_p, \quad \text{for all } u \in T_p\Sigma,
\]
where $\overline{\nabla}$ denotes the Levi-Civita connection on $M$. The shape operator is diagonalizable and we denote by $\{e_1, e_2\}$ a positively oriented orthonormal frame such that $S e_i = \kappa_i e_i$ for $i \in \{1,2\}$. The functions $\kappa_1$ and $\kappa_2$ defined on $\Sigma$ are the \emph{principal curvatures} of the immersion $f$ and they are regular outside umbilical points. Recall that the mean curvature is given by $H = \tfrac{1}{2} \operatorname{tr}(S) = \tfrac{1}{2}(\kappa_1 + \kappa_2)$. If we take the opposite unit normal $-N$ along $f$, the mean curvature changes sign. It is important to note that, if we prescribe the principal curvatures $\kappa_1$ and $\kappa_2$, then, outside minimal points, there is only one choice of $N$ that realizes these principal curvatures.

\subsection{Bianchi-Cartan-Vr\u{a}nceanu spaces}\label{subsec:description Ekt spaces}

We denote these spaces by $\mathbb{E}(\kappa,\tau)$ where $\kappa$ and $\tau$ are real constants. They can be described as unit Killing submersions with constant bundle curvature $\tau$ over $\mathbb{M}^2(\kappa)$, the complete simply connected Riemannian surface of constant Gaussian curvature $\kappa$. In other words, there exists a Riemannian submersion $\pi : \mathbb{E}(\kappa,\tau) \to \mathbb{M}^2(\kappa)$ together with a unit Killing vector field $\xi$ such that $\mathrm{d} \pi(\xi) = 0$. These spaces are characterized as the only complete, simply connected, homogeneous Riemannian $3$-manifolds with a $4$-dimensional isometry group when $\kappa \neq 4\tau^2$ (see~\cite{Man14}).

To describe the main properties of $\mathbb{E}(\kappa,\tau)$, we introduce the Cartan model (see~\cite{Dan07}). Consider  
\[
\Omega_\kappa = \{(x,y)\in \mathbb{R}^2 : \lambda_\kappa (x,y) > 0\}, \text{ where } 
\lambda_\kappa(x,y) = \left(1 + \tfrac{\kappa}{4}(x^2 + y^2)\right)^{-1}.
\]  
Note that the Riemannian metric $\lambda_\kappa^2(\mathrm{d} x^2 + \mathrm{d} y^2)$ has constant Gaussian curvature $\kappa$ on $\Omega_\kappa$. Then, for each $\kappa, \tau \in \mathbb{R}$, the space $\mathbb{E}(\kappa,\tau)$ is locally isometric to the product $\Omega_\kappa \times \mathbb{R} \subset \mathbb{R}^3$, endowed with the Riemannian metric  
\begin{equation}\label{eqn:CartanMetric}
	\mathrm{d} s^2_{\kappa,\tau} = \lambda_\kappa^2(\mathrm{d} x^2 + \mathrm{d} y^2) +  (\mathrm{d} z +  \tau \lambda_\kappa(y \mathrm{d} x - x \mathrm{d} y))^2.
\end{equation}  

The projection $(x,y,z)\mapsto (x,y)$ defines the unit Killing submersion onto $(\Omega_\kappa, \lambda_\kappa^2(\mathrm{d} x^2 + \mathrm{d} y^2))$, whose fibers are the integral curves of the unit Killing vector field $\partial_z$; see~\cite{Dan07, LerMan, Man14} for further details on Killing submersions.

In what follows, we assume that the spaces $\mathbb{E}(\kappa,\tau)$ are oriented so that  
\begin{equation}\label{eqn:KillingFrame}
	\begin{aligned}
		V_1&= \lambda^{-1}_{\kappa}\partial_x-\tau y\partial_z,&\quad 
		V_2&=\lambda^{-1}_{\kappa}\partial_y+\tau x\partial_z,&\quad 
		V_3 &= \partial_z,
	\end{aligned}
\end{equation}  
defines a positively oriented orthonormal frame.

If $\kappa \leq 0$, then the model is complete and globally isometric to the corresponding space $\mathbb{E}(\kappa,\tau)$. If $\kappa > 0$, the model is incomplete, but it is globally isometric to the universal cover of $\mathbb{E}(\kappa,\tau)$ minus a vertical fiber $\pi^{-1}(p)$ for some $p \in \mathbb{M}^2(\kappa)$.

\subsubsection{The compatibility equations in $\mathbb{E}(\kappa,\tau)$}
We define the \emph{angle function} $\nu_3:\Sigma \rightarrow \mathbb{R}$ given by $\nu_3 = \langle \xi, N\rangle$ and the vector field $T_3\in \mathfrak{X}(\Sigma)$ given by the tangential part of the unit Killing vector field $T_3 = \xi^\top = \xi - \nu_3 N $. Observe that we have $\nu_3^2 \leq 1$.

Given an oriented Riemannian surface $\Sigma$ and an isometric immersion $f:\Sigma \rightarrow \mathbb{E}(\kappa,\tau)$, we refer to the tuple $(J, S, T_3, \nu_3)$ as the \emph{fundamental data} of the immersion. We now state the following theorem, which collects the compatibility equations for surfaces isometrically immersed into unit Killing submersions over $\mathbb{M}^2(\kappa)$.
\begin{proposition}[\cite{Dan07,Dan09}]\label{prop:ecuaciones de compatibilidad ekt}
	Let $(\Sigma, \mathrm{d} s^2, J)$ be a simply connected oriented Riemannian surface. Let $S: T\Sigma \rightarrow T\Sigma$ be a field of self-adjoint endomorphisms, $T_3$ a vector field on $\Sigma$ and $\nu_3$ a function on $\Sigma$. Then there exists an isometric immersion $f: \Sigma \rightarrow \mathbb{E}(\kappa,\tau)$ with shape operator $S$ such that the unit Killing field $\xi$ satisfies $\xi = T_3 + \nu_3 N$, if and only if, the tuple $(J, S, T_3, \nu_3)$ satisfies the following equations for all vector fields $X, Y \in \mathfrak{X}(\Sigma)$:
	\begin{enumerate}
		\item $K= \det(S)+\tau^2+(\kappa-4\tau^2)\nu_3^2$, \hfill \emph{(Gauss eqn.)}
		\item $\nabla_XSY-\nabla_YSX-S[X,Y]= (\kappa-4\tau^2)(\langle Y,T_3\rangle X-\langle X,T_3\rangle Y)\nu_3$, \hfill \emph{(Codazzi eqn.)}
		\item $\nabla_X T_3=\nu_3(SX-\tau JX)$,
		\item $\nabla\nu_3=-ST_3-\tau JT_3$,
		\item $||T_3||^2 = 1-\nu_3^2$,
	\end{enumerate}
	where $\nabla$ denotes the Levi-Civita connection of $\Sigma$.
	If these conditions hold, then the immersion is unique up to ambient orientation-preserving isometries that also preserve $\xi$.
\end{proposition}

\begin{remark}\label{remark:isometries and fundamental data}
	The isometry group $\mathrm{Iso}(\mathbb{E}(\kappa,\tau))$ has dimension $4$. Any orientation preserving isometry leaves the fundamental data unchanged. However, certain geometric transformations can change the fundamental data while giving the same immersion (see \cite{CMT, Dan07, GMM}).
	\begin{enumerate}
		\item Changing the sign of the normal vector field $N$ gives $(J, S, T_3, \nu_3) \rightarrow (-J, -S, T_3, -\nu_3)$.
		\item Composing with an orientation-preserving isometry that reverses the orientation of the fibers yields the transformation $(J, S, T_3, \nu_3) \rightarrow (J, S, -T_3, -\nu_3)$.
		\item When $\tau = 0$, an isometry that reverses the orientation of the surface while preserving the orientation of the fibers gives $(J, S, T_3,\nu_3) \rightarrow (-J, S, T_3, \nu_3)$. \hfill $\blacksquare$
	\end{enumerate}
\end{remark}
We recall the following expressions which will be useful later. From equation \emph{(3)} in Proposition~\ref{prop:ecuaciones de compatibilidad ekt}, it is straightforward to verify that
\begin{equation}\label{eq:divergences T JT}
	\operatorname{div}(T_3) = 2 H \nu_3, \qquad \quad \operatorname{div}(J T_3) = 2 \tau \nu_3
\end{equation}
\noindent where $\mathrm{div}$ is the divergence operator defined on $\Sigma$. Taking into account the identity $S^2 = 2 H S - \det(S)\;\mathrm{Id}$, that follows from Cayley--Hamilton theorem applied to the linear map $S$, we obtain the following expression for the norm of $\nabla\nu_3$
\begin{equation}\label{eq:norm nabla nu}
	\begin{aligned}
		||\nabla\nu_3||^2 &= \langle - ST_3 - \tau JT_3, - ST_3 - \tau JT_3\rangle = \langle S^2T_3, T_3\rangle + 2\tau \langle ST_3, JT_3\rangle +  \tau^2 (1-\nu_3^2)\\
		&=2 H\langle ST_3, T_3\rangle - \det(S)  (1-\nu_3^2) + 2\tau \langle ST_3, JT_3\rangle +  \tau^2 (1-\nu_3^2).
	\end{aligned}
\end{equation}
\noindent Taking trace in Codazzi equation fixing $X\in \mathfrak{X}(\Sigma)$, we have
\begin{equation}\label{eq:trace skew-symmetric operators}
	2 \langle \nabla H, X\rangle - \mathrm{div}(SX) + \sum_{i=1}^2 \langle \nabla_{e_i}X, Se_i\rangle = - (\kappa - 4  \tau^2)\nu_3 \langle X, T_3\rangle.
\end{equation}
\noindent We can now present a reduction of the compatibility equations.

\begin{lemma}\label{th:adapted theorem}
	Let $(\Sigma, \mathrm{d} s^2, J)$ be an oriented Riemannian surface. Let $S: T\Sigma \to T\Sigma$ be a self-adjoint endomorphism field, $T_3$ a vector field on $\Sigma$ and let $\nu_3$ be a function on $\Sigma$ with $\nu_3^2 \leq 1$. Assume moreover that the set of points where $\nu_{3}^{2} = 1$ has empty interior (see Remark~\ref{rmk:slice}). Then an immersion with these data is uniquely determined by the equations
	\begin{enumerate}[label=$($\textsc{\roman*}$)$]
		\item $\nabla_X T_3 =  \nu_3(SX - \tau JX)$, for every $X\in \mathfrak{X}(\Sigma)$,
		\item $\nabla\nu_3 = - ST_3 - \tau JT_3$,
		\item $\Delta \nu_3  =-(\kappa - 4 \tau^2)(1-\nu_3^2)\nu_3 -   (\kappa_1^2 + \kappa_2^2)\nu_3 - 2 \tau^2\nu_3 - 2 \langle \nabla H, T_3\rangle$,
		\item $||T_3||^2 = 1-\nu_3^2$.
	\end{enumerate}
\end{lemma}
\begin{proof}
	We prove that this system of equations is equivalent to equations \emph{(1)}--\emph{(5)} 
	of Proposition~\ref{prop:ecuaciones de compatibilidad ekt}. 
	Equations \textsc{(i)}, \textsc{(ii)} and \textsc{(iv)} correspond respectively to 
	\emph{(3)}, \emph{(4)} and \emph{(5)}. Therefore, assuming these three equations, 
	it will be enough to show two facts: that \textsc{(iii)} is equivalent to \emph{(2)}, and 
	that \textsc{(i)}–\textsc{(iv)} imply~\emph{(1)}.
	
	Assume first that \emph{(2)} holds. Then equality \textsc{(iii)} follows from 
	\eqref{eq:divergences T JT} (itself a consequence of \emph{(1)}--\emph{(5)}) and from 
	the identity $|\sigma|^2 = \kappa_1^2 + \kappa_2^2$. Conversely, suppose that 
	\textsc{(iii)} holds, and define the skew-symmetric tensors
	\begin{align*}
		&\mathcal{T}_S (X,Y) = \nabla_X SY - \nabla_Y SX - S[X,Y],\\
		&\Theta(X,Y) =  (\kappa - 4 \tau^2)\nu_3 
		\big(\langle Y, T_3\rangle X - \langle X, T_3\rangle Y\big).
	\end{align*}
	\noindent The trace of $\mathcal{T}_S$ is given by the expression on the left-hand side of 
	equation~\eqref{eq:trace skew-symmetric operators}. Substituting equations 
	\textsc{(i)} and \textsc{(iii)} into this expression yields
	\begin{align*}
		\mathrm{tr}\; \mathcal{T}_S (T_3, \cdot) 
		&=  2 \langle \nabla H, T_3\rangle - \mathrm{div}(ST_3) 
		+ \sum_{i=1}^2 \langle \nabla_{e_i}T_3, Se_i\rangle 
		= - (\kappa - 4  \tau^2)\nu_3 (1-\nu_3^2),\\
		\mathrm{tr}\; \mathcal{T}_S (JT_3, \cdot) 
		&= 2 \langle \nabla H, JT_3\rangle - \mathrm{div}(SJT_3) 
		+ \sum_{i=1}^2 \langle \nabla_{e_i}JT_3, Se_i\rangle = 0.
	\end{align*}
	These two identities determine a unique $1$-form satisfying  $\mathrm{tr}\; \mathcal{T}_S (X,\cdot) = - (\kappa - 4 \tau^2)\nu_3 \langle X, T_3\rangle$.	We now note that $\mathcal{T}_S(X,Y)$ is completely determined by its trace. Indeed, 
	in the orthonormal frame $\{e_1,e_2\}$ we have
	\begin{align*}
		\mathcal{T}_S (e_1, e_2) 
		&= \langle \mathcal{T}_S (e_1, e_2), e_1\rangle e_1 
		+ \langle \mathcal{T}_S (e_1, e_2), e_2\rangle e_2 
		= \mathrm{tr}\; \mathcal{T}_S (e_1, \cdot)\, e_2 
		- \mathrm{tr}\; \mathcal{T}_S(e_2, \cdot)\, e_1\\
		&=  (\kappa - 4 \tau^2)\nu_3 
		\big(\langle e_1, T_3\rangle e_2 - \langle e_2, T_3\rangle e_1\big) =  (\kappa - 4 \tau^2)\nu_3 JT_3.
	\end{align*}
	\noindent By antisymmetry and bilinearity, this expression determines $\mathcal{T}_S$, and hence $\mathcal{T}_S = \Theta$. Thus \emph{(2)} holds.
	
	Now we show that Gauss’ equation \emph{(1)} follows from 
	\textsc{(i)}--\textsc{(iv)}. Using the Böchner formula as in \cite[p.~4]{TorrUrb10}, we have
	\[
	\mathrm{div}(\nabla_{T_3} T_3) 
	= K (1-\nu_3^2) + 2\,\mathrm{div}(\nu_3 H T_3) - 4 H^2 \nu_3^2 
	+ \nu_3^2 (\kappa_1^2 + \kappa_2^2) - 2 \tau^2 \nu_3^2.
	\]
	Expanding the left-hand side using \textsc{(ii)} gives
	\begin{align*}
		\mathrm{div}(\nabla_{T_3} T_3) 
		&= \mathrm{div}\big( \nu_3 (ST_3 - \tau JT_3)\big) = -\mathrm{div}\big(\nu_3(\nabla\nu_3 + 2\tau JT_3)\big) \\ 
		&= - |\nabla\nu_3|^2 -  \nu_3 \Delta\nu_3 
		- 2 \tau \langle \nabla\nu_3, JT_3\rangle - 4\tau^2\nu_3^2.
	\end{align*}
	Using equations~\eqref{eq:norm nabla nu} (which follows from	\textsc{(i)}--\textsc{(iv)}) and \textsc{(iii)}, and equating both sides, we conclude that
	\[
	K(1-\nu_3^2) 
	= (\det(S) + \tau^2)(1-\nu_3^2) 
	+ (\kappa - 4\tau^2)\nu_3^2 (1-\nu_3^2).
	\]
	Since we assume $\nu_3^2 \neq 1$ on a dense open subset of $\Sigma$, continuity implies that Gauss’ equation holds everywhere.
\end{proof}
\begin{remark}\label{rmk:slice}
	Excluding the case where $\nu_3^{2} = 1$ on an open set is not a serious restriction. In that situation, the tangent plane of the immersion lies in the distribution orthogonal to the unit Killing field, and this distribution is integrable if and only if $\tau = 0$. Hence we are just excluding the case of immersions contained in a slice $\mathbb{M}^{2}(\kappa)\times\{t_{0}\}$ for some $t_{0}\in \mathbb{R}$. \hfill $\blacksquare$
\end{remark}
\subsection{$\mathrm{Sol}_3$-type geometries}\label{subsec:description Sol3 type spaces} We now turn to the family of Riemannian $3$-manifolds denoted by $\mathrm{Sol}_3$ which form a one-parameter family of unimodular metric Lie groups. Each of these manifolds is diffeomorphic to $\mathbb{R}^3$, endowed with a Lie group structure parametrized by $\mu > 0$ and given by
\[
(x,y,z)\star (x',y',z') = ( x + x' e^{-\mu z}, \; y + y' e^{\mu z}, \; z + z').
\]
\noindent Given a point $p \in \mathrm{Sol}_3$, the \emph{left translation} by $p$ is the map $L_p: \mathrm{Sol}_3\to \mathrm{Sol}_3$ defined by $L_p (g) = p\star g$. This defines a free and transitive action of $\mathrm{Sol}_3$ on itself, which allows us to endow $\mathrm{Sol}_3$ with a standard left-invariant Riemannian metric. In coordinates, this metric (see~\cite[Eq. 2.10]{CMS}) is given by
\[\mathrm{d} s^2 = \cosh(2\mu z) (\mathrm{d} x^2 + \mathrm{d} y^2) + \sinh (2\mu z) (\mathrm{d} x \mathrm{d} y + \mathrm{d} y \mathrm{d} x) + \mathrm{d} z^2.\]
A global orthonormal frame of left-invariant vector fields is given by
\[E_1 = \cosh(\mu z) \partial_x - \sinh(\mu z)\partial_y,\quad  E_2 = -\sinh(\mu z)\partial_x + \cosh(\mu z)\partial_y,\quad  E_3 = \partial_z\]
\noindent which satisfy the commutation relations $[E_1, E_2] = 0$, $[E_2, E_3] = \mu E_1$, $[E_3, E_1] = - \mu E_2$. These vector fields determine a natural orientation on $\mathrm{Sol}_3$ by declaring the frame $\{E_1, E_2, E_3\}$ to be positively oriented. They also encode the underlying Lie group structure (see~\cite{MeeksPerez12, Milnor}).

The isometry group $\mathrm{Iso}(\mathrm{Sol}_3)$ is $3$-dimensional and has $8$ connected components. By construction, left translations are isometries and they generate the identity component. Indeed, the components are in correspondence with elements of the isotropy group. This group is isomorphic to the dihedral group $D_4$ and is generated by the following isometries fixing the origin:
\begin{equation}\label{eq:isotropy Sol3}
	\Psi_1 (x,y,z) = (y, -x, -z), \quad \Psi_2(x,y,z) = (x,-y,-z).
\end{equation}

\subsubsection{The compatibility equations in $\mathrm{Sol}_3$.} Given an immersion $f:\Sigma \to \mathrm{Sol}_3$, we have a left-invariant Gauss map which is given by composing the unit normal $N$ along $f$ with the inverse of the left translation at each point $p$ as $(\mathrm{d} L_p)^{-1}\circ N_p$. This map takes values in the unit sphere of the Lie algebra $\mathfrak{sol}_3$. By slight abuse of notation, we will denote this left-invariant Gauss map simply by $N$. With respect to the left-invariant frame $\{E_1, E_2, E_3\}$, we define the angle functions $\nu_\alpha = \langle E_\alpha, N\rangle$ for $\alpha\in\{1,2,3\}$, this is, we have $N = \sum_{\alpha=1}^3 \nu_\alpha E_\alpha$. We also define the vector fields $T_1$, $T_2$, $T_3\in \mathfrak{X}(\Sigma)$ as the tangent components of the vector fields $E_1$, $E_2$, $E_3$ given by $T_\alpha = E_\alpha^\top = E_\alpha - \nu_\alpha N$ for $\alpha\in\{1,2,3\}$. In these spaces, we refer to the tuple $(J, T_1, T_2, T_3, \nu_1, \nu_2, \nu_3)$ as the fundamental data of the immersion. According to \cite[Eqns.~3.2 and 3.3]{CMS}, the following algebraic identities are satisfied
\begin{equation}\label{eq:algebraic relations}
	\begin{aligned}
		&\langle T_\alpha, T_\beta\rangle = \delta_{\alpha\beta}-\nu_\alpha \nu_\beta,\quad \langle JT_\alpha, T_{\alpha+1}\rangle = \nu_{\alpha-1},\\
		&\nu_{\alpha-1} T_{\alpha+1} - \nu_{\alpha+1} T_{\alpha-1} = JT_\alpha, \quad \sum_{\alpha=1}^3 \nu_\alpha^2 = 1,
	\end{aligned}
\end{equation}
\noindent for every $\alpha$, $\beta\in\{1,2,3\}$, where indexes are considered cyclically.

It is remarkable that in Subsection~\ref{subsec:description Ekt spaces}, the operator $S$ was included as part of the fundamental data of the immersion. In the present case, we work with the theory of isometric immersions developed in~\cite{CMS}, where it is shown that the shape operator $S$ can be completely determined in terms of the data $\nu_\alpha$, $T_\alpha$ for $\alpha \in \{1,2,3\}$ as follows:
\begin{equation}\label{thm:fundamental:eqn1}
	SX = \mu \big(\langle T_2, X\rangle JT_2 - \langle T_1, X\rangle JT_1\big)
	- \sum_{\alpha=1}^3 \langle \nabla \nu_\alpha, X\rangle T_\alpha.
\end{equation}
\noindent In addition, defining $\zeta = \mu (\nu_1^2 - \nu_2^2)$, the operator $S$ is self-adjoint if and only if the equation $\sum_{\alpha = 1}^3 \langle \nabla \nu_\alpha, JT_\alpha\rangle + \zeta = 0$ holds and the mean curvature is expressed as $\sum_{\alpha = 1}^3 \langle \nabla \nu_\alpha, T_\alpha\rangle = - 2H$. Therefore, for convenience of notation we may use $S$, but we do not include it as part of the fundamental data since it is completely determined by the rest.

\begin{proposition}\label{prop:compatibility Sol3}
	\emph{\cite[Proposition~3.1]{CMS}} Let $(\Sigma, \mathrm{d} s^2, J)$ be an oriented Riemannian surface isometrically immersed into $\mathrm{Sol}_3$. The tuple $(J,T_1,T_2,T_3,\nu_1,\nu_2,\nu_3)$ satisfies the following relations:
	\begin{enumerate}[label=\emph{(\roman*)}]
		\item $K = \det(S) - \mu^2 + 2 \mu^2 \nu_3^2$,\hfill \emph{(Gauss eqn.)}
		\item $\nabla_X SY - \nabla_Y SX - S[X, Y] = 2 \mu^2 \nu_3 \left(\langle Y, T_3\rangle X - \langle X, T_3\rangle Y\right)$, \hfill \emph{(Codazzi eqn.)}
		\item $\nabla_X T_1=\nu_1 SX - \mu \langle X,T_2\rangle T_3$,
		
		\noindent $\nabla_X T_2=\nu_2 SX- \mu \langle X,T_1\rangle T_3$,
		
		\noindent $\nabla_X T_3=\nu_3SX+ \mu \langle X,T_2\rangle T_1 + \mu \langle X,T_1\rangle T_2$, for all $X\in\mathfrak{X}(\Sigma)$,
		
		\item $\nabla\nu_1=-ST_1- \mu\nu_3 T_2$,
		
		\noindent $\nabla\nu_2=-ST_2- \mu\nu_3 T_1$,
		
		\noindent $\nabla\nu_3=-ST_3 + \mu \nu_1 T_2 +\mu \nu_2 T_1$.
	\end{enumerate}
\end{proposition}

\begin{remark}\label{remark:isometries and fundamental data in Sol3}
	Every left translation preserves the fundamental data. However, isometries outside this component do not preserve the left-invariant Gauss map. We describe how the left-invariant Gauss map and the orientation of the surface are affected by the generators of the isotropy group $\Psi_1$, $\Psi_2$.
	\begin{enumerate}
		\item $\Psi_1$ reverses the orientation and induces the transformation $(J, \nu_1, \nu_2, \nu_3) \longmapsto (-J, \nu_2, -\nu_1, -\nu_3)$,
		\item $\Psi_2$ preserves the orientation and induces the transformation $(J, \nu_1, \nu_2, \nu_3) \longmapsto (J, \nu_1, -\nu_2, -\nu_3)$.
		\item Changing the sign of $N$ yields $(J, \nu_1, \nu_2, \nu_3) \longmapsto (-J, -\nu_1, -\nu_2, -\nu_3)$. \hfill $\blacksquare$
	\end{enumerate}
\end{remark}
Similarly to~\eqref{eq:divergences T JT}, in $\mathrm{Sol}_3$ we have the following identities derived from equations~(\romannumeral 3) in Proposition~\ref{prop:compatibility Sol3}:
\begin{equation}\label{eq:divergences T JT in Sol}
	\text{div}(T_3) = 2 H\nu_3 - 2\mu \nu_1 \nu_2, \qquad \qquad \text{div}(JT_3) = 0.
\end{equation}
\noindent In the context of $\mathrm{Sol}_3$, the vector fields defined in \cite[Equation 3.36]{CMS} and given by
\begin{equation}\label{eq:X alpha}
	\begin{aligned}
		X_1 &= J\nabla\nu_1 + \nu_3\nabla\nu_2 - \nu_2 \nabla\nu_3,\\
		X_2 &= J\nabla\nu_2 + \nu_1\nabla\nu_3 - \nu_3\nabla\nu_1,\\
		X_3 &= J\nabla\nu_3 + \nu_2 \nabla\nu_1 - \nu_1 \nabla\nu_2\\
	\end{aligned}
\end{equation}
\noindent play an important role in our arguments. First of all, notice that, by~\cite[Prop.~3.14]{CMS}, any surface for which $X_1 = X_2 = X_3 = 0$ must be a left coset of a $2$-dimensional subgroup of $\mathrm{Sol}_3$ (these surfaces are studied in~\cite[Ex.~3.4]{MeeksPerez12} in a different model). In this case, $\nu_1$, $\nu_2$, and $\nu_3$ are constants, satisfying the additional condition $\nu_1^2 = \nu_2^2$, and $H = 0$. Since these surfaces are extrinsically homogeneous, both their Gauss curvature $K$ and their extrinsic curvature $\det(S)$ are constant.

Leaving these special cases aside, we now consider immersions whose 
left-invariant Gauss map is non-constant on every open set. Consequently, 
$\nu_\alpha^{2} \neq 1$ for all $\alpha \in \{1,2,3\}$ on a dense open subset of~$\Sigma$, so we may work on an open region where the tangent fields $T_1$, $T_2$, and $T_3$ are all non-vanishing. We can express $T_{1}$ and $T_{2}$ in terms of the angle functions and the vector fields $T_{3}$ and $JT_{3}$ as
\begin{equation}\label{eq:T1 T2 en T3}
		(1-\nu_3^2)T_1 = -\nu_1 \nu_3 T_3 + \nu_2 JT_3,\qquad \qquad (1-\nu_3^2)T_2 = - \nu_2 \nu_3 T_3 - \nu_1 JT_3.
\end{equation}
From these relations, one verifies directly that, after rewriting $\sum_{\alpha = 1}^{3} \langle \nabla \nu_{\alpha}, JT_{\alpha} \rangle 
= -\zeta$ and $\sum_{\alpha = 1}^{3} \langle \nabla \nu_{\alpha}, T_{\alpha} \rangle 
= - 2H$ in terms of $T_{3}$ and $JT_{3}$ using~\eqref{eq:T1 T2 en T3}, both expressions are equivalent to $X_{3} = -2H\, JT_{3} + \zeta\, T_{3}$.

The following result is a technical Lemma that help us to prove Theorem~\ref{thm:compatibility equations Sol3}.
\begin{lemma}\label{lema:nuevo}
	Let $(\Sigma, \mathrm{d}s^{2})$ be a Riemannian surface with fundamental data
	$(J, T_{1}, T_{2}, T_{3}, \nu_{1}, \nu_{2}, \nu_{3})$ satisfying the algebraic
	relations~\eqref{eq:algebraic relations} and such that $X_{1}, X_{2}, X_{3}$ are nonzero on a dense open set, and let $H$ be a function on~$\Sigma$. If $X_{3}$ given by~\eqref{eq:X alpha}, satisfies $X_{3} = -2H\, JT_{3} + \zeta\, T_{3}$, then the following hold automatically:
	\begin{enumerate}
		\item[(\textit{1.})] $X_{1} = -2H\, JT_{1} + \zeta\, T_{1}, \quad
		X_{2} = -2H\, JT_{2} + \zeta\, T_{2}$;
		\item[(\textit{2.})] the equalities \textsc{(iv)} in Proposition~\ref{prop:compatibility Sol3}.
	\end{enumerate}
\end{lemma}
\begin{proof}
We just prove the first identity in (\textit{1.}), since the second one is completely analogous.

Using the algebraic relations, we can write $T_1$ and $JT_1$ in terms of $T_3$ and $JT_3$ via~\eqref{eq:T1 T2 en T3}, and after grouping the resulting terms with respect to $T_3$ and $JT_3$ and using the condition $X_3 = - 2HJT_3 + \zeta T_3$, we obtain
\begin{equation}\label{new equation}
	-2HJT_1 + \zeta T_1 
	= \frac{-\nu_1\nu_3}{1-\nu_3^2} X_3 + \frac{\nu_2}{1-\nu_3^2}JX_3.
\end{equation}
Substituting the definition of $X_3$, the right-hand side of~\eqref{new equation} can be written as
\begin{equation}\label{eq:ecuacion}
	\frac{1}{1-\nu_3^2}\Big(
	-\nu_1 \nu_3 J\nabla\nu_3 
	- \nu_2 \nabla\nu_3 
	- \nu_1 \nu_2 \nu_3 \nabla \nu_1 
	+ \nu_1^2 \nu_3 \nabla\nu_2 
	+ \nu_2^2 J\nabla\nu_1 
	- \nu_1 \nu_2 J\nabla\nu_2
	\Big).
\end{equation}
\noindent We now look at the terms inside the parentheses in~\eqref{eq:ecuacion}, grouping them according to whether the operator $J$ appears or not. By differentiating $\nu_1^2 + \nu_2^2 + \nu_3^2 = 1$, we find the algebraic relation $\nu_1\nabla\nu_1 = - \nu_2\nabla\nu_2 - \nu_3\nabla\nu_3$. Concerning the terms in~\eqref{eq:ecuacion} where $J$ appears, namely $\nu_2^2 J\nabla\nu_1 + \nu_1(-\nu_2 J\nabla\nu_2 - \nu_3 J\nabla\nu_3)$, this reduces to $(1-\nu_3^2)J\nabla\nu_1$. For the remaining terms in the parentheses, substituting $\nu_1\nabla\nu_1$ using the algebraic relation gives $(1-\nu_3^2)(\nu_3\nabla\nu_2 - \nu_2\nabla\nu_3)$.

In conclusion, equation~\eqref{eq:ecuacion} reduces to $J\nabla\nu_1 + \nu_3 \nabla\nu_2 - \nu_2\nabla\nu_3 = X_1$, which proves the first item in the statement.

Regarding item (\textit{2.}), we will only prove the equation for $\nabla\nu_3$ in Proposition~\ref{prop:compatibility Sol3}, since the other two follow in exactly the same way. Note that, using the expression~\eqref{thm:fundamental:eqn1} for $S$, the equation is equivalent to
\[
\nabla\nu_3 = \bigg(\mu \langle T_1, T_3\rangle JT_1 - \mu \langle T_2, T_3\rangle JT_2 
+ \sum_{\alpha=1}^3 \langle \nabla\nu_\alpha, T_3\rangle T_\alpha\bigg) 
+ \mu \nu_1 T_2 + \mu \nu_2 T_1,
\]
\noindent where the part inside the parentheses corresponds to $ST_3$. Let us denote the right-hand side by $W$. Taking inner products with $T_3$ and $JT_3$ (both nonzero since $\nu_3^2 \neq 1$), we will obtain $\langle \nabla\nu_3, T_3\rangle$ and $\langle \nabla\nu_3, JT_3\rangle$, respectively.

We take inner products with $T_3$, and using the algebraic relations several terms cancel out, so we obtain
\[
\langle W, T_3\rangle 
= \sum_{\alpha=1}^3 \langle \nabla \nu_\alpha, T_3\rangle \langle T_\alpha, T_3\rangle 
= -\nu_1 \nu_3 \langle \nabla \nu_1, T_3\rangle 
- \nu_2 \nu_3 \langle \nabla \nu_2, T_3\rangle 
+ (1-\nu_3^2)\langle \nabla \nu_3, T_3\rangle.
\]
\noindent Moreover, using $\sum_{\alpha=1}^3 \nu_\alpha \nabla \nu_\alpha = 0$, 
the expression simplifies and the final result is exactly 
$\langle \nabla \nu_3, T_3\rangle$.  
A completely analogous argument applies to the case of $JT_3$.
\end{proof}
\begin{remark}\label{obs:equivalencia III}
	Under the hypotheses of the previous lemma, the equations \textsc{(iii)} of 
	Proposition~\ref{prop:compatibility Sol3} are respectively equivalent to
	\begin{align*}
		&\langle \nabla_X T_1, JT_1 \rangle 
		= \nu_1 \langle SX, JT_1 \rangle + \mu \nu_2 \langle X, T_2 \rangle, \\
		&\langle \nabla_X T_2, JT_2 \rangle
		= \nu_2 \langle SX, JT_2 \rangle - \mu \nu_1 \langle X, T_1 \rangle, \\
		&\langle \nabla_X T_3, JT_3 \rangle
		= \nu_3 \langle SX, JT_3 \rangle 
		+ \mu \langle X, \nu_2 T_2 - \nu_1 T_1 \rangle,
		\qquad \text{for all $X \in \mathfrak{X}(\Sigma)$}.
	\end{align*}
	
	Indeed, that \textsc{(iii)} implies these equations is immediate by multiplying 
	each $\nabla_X T_\alpha$ by $JT_\alpha$, using that $JT_\alpha$ does not vanish 
	because $\nu_\alpha^{2}\neq 1$. Conversely,  under the 
	hypotheses of Lemma~\ref{lema:nuevo}, we can recover the equations for 
	$\langle \nabla_X T_\alpha, T_\alpha\rangle$. Differentiating the algebraic relations $\langle T_\alpha, T_\alpha \rangle = 1 - \nu_\alpha^{2}$, we obtain $\langle \nabla_X T_\alpha, T_\alpha \rangle
	= -\nu_\alpha \langle \nabla \nu_\alpha, X \rangle$. Since we are in the setting of Lemma~\ref{lema:nuevo}, the equations 
	\textsc{(iv)} in Proposition~\ref{prop:compatibility Sol3} are satisfied and substituting, for instance in the case $\alpha=3$ (the cases $\alpha\in\{1,2\}$ are 
	handled similarly), we obtain
	\begin{align*}
		\langle \nabla_X T_3, T_3 \rangle
		&= \nu_3 \,\big\langle ST_3 
		+ \mu \langle T_3, T_2 \rangle T_1
		+ \mu \langle T_3, T_1 \rangle T_2 ,\, X \big\rangle \\
		&= \nu_3 \,\big\langle SX 
		+ \mu \langle X, T_2 \rangle T_1
		+ \mu \langle X, T_1 \rangle T_2 ,\, T_3 \big\rangle.
		\tag*{$\blacksquare$}
	\end{align*}
\end{remark}

We now present an adapted version of the compatibility equations in $\mathrm{Sol}_3$ on which the computations for the Bonnet problem will be based.

\begin{theorem}\label{thm:compatibility equations Sol3}
	Let $(\Sigma, \mathrm{d}s^{2}, J)$ be an oriented, simply connected Riemannian 
	surface. Consider fundamental data $(J, T_1, T_2, T_3, \nu_1, \nu_2, \nu_3)$ satisfying the algebraic relations~\eqref{eq:algebraic relations}, and assume 
	that $X_1$, $X_2$ and $X_3$ are nonzero on a dense open set. Fix a function $H$ 
	on $\Sigma$. Then there exists a unique isometric immersion of $\Sigma$ into 
	$\mathrm{Sol}_3$ such that $T_1$, $T_2$, $T_3$ are the tangent components and $H$ 
	is the mean curvature if and only if
	\begin{itemize}
		\item[$(a)$] $X_3$, given by~\eqref{eq:X alpha}, coincides with $-2H\, JT_3 + \zeta\, T_3$ and
		\item[$(b)$] $\langle \nabla_X T_3, JT_3 \rangle
		= \nu_3 \big(\nu_1 \langle \nabla \nu_2, X \rangle
		- \nu_2 \langle \nabla \nu_1, X \rangle\big)
		+ \mu (1 - \nu_3^{2}) 
		\langle X, \nu_2 T_2 - \nu_1 T_1 \rangle$, for every $X\in \mathfrak{X}(\Sigma)$.
	\end{itemize}
	This immersion is unique up to left translation by an element of $\mathrm{Sol}_3$.
\end{theorem}
\begin{proof}
	For the first implication, note that the fact that $S$ is a self-adjoint operator with 
	mean curvature $H$ already gives $(a)$, as explained above. 
	Moreover, since the shape operator is determined by~\eqref{thm:fundamental:eqn1}, 
	we multiply the expression for $\nabla_X T_3$ in item~\textsc{(iii)} of 
	Proposition~\ref{prop:compatibility Sol3} by $JT_3$, and substitute $S$ using 
	\eqref{thm:fundamental:eqn1}, which yields the identity in $(b)$.
	
	Conversely, assume that $(a)$ and $(b)$ hold.  
	By~\cite[Thm. 3.6]{CMS}, it is enough to check that the equations in 
	\textsc{(iii)} of Proposition~\ref{thm:compatibility equations Sol3} are satisfied for the endomorphism $S$ defined by~\eqref{thm:fundamental:eqn1}, which must be 
	self-adjoint. First $S$ is self-adjoint and has mean curvature 
	$H$ by $(a)$. In addition, the conditions in \textsc{(iv)} of 
	Proposition~\ref{prop:compatibility Sol3} hold automatically by Lemma~\ref{lema:nuevo}, 
	and the equation for $\nabla_X T_3$ appearing in \textsc{(iii)} of 
	Proposition~\ref{thm:compatibility equations Sol3} is equivalent to $(b)$, according 
	to Remark~\ref{obs:equivalencia III}.  
	With this, we can verify that the expressions for $\nabla_X T_1$ and $\nabla_X T_2$ 
	in equation~\textsc{(iii)} of Proposition~\ref{thm:compatibility equations Sol3} 
	also hold. Indeed, the assumption $\nu_\alpha^2 \neq 1$ for 
	$\alpha \in \{1,2,3\}$ allows us to write $T_1$ and $T_2$ in terms of the angle 
	functions and the fields $T_3$ and $JT_3$, as in~\eqref{eq:T1 T2 en T3}. Hence,
	\begin{align*}
		\nabla_X T_1 &= \nabla_X \left(
		\frac{-\nu_1 \nu_3}{1-\nu_3^2} T_3 
		+ \frac{\nu_2}{1-\nu_3^2} JT_3 \right),\\
		\nabla_X T_2 &= \nabla_X \left(
		\frac{-\nu_2 \nu_3}{1-\nu_3^2} T_3 
		- \frac{\nu_1}{1-\nu_3^2} JT_3 \right).
	\end{align*}
	Expanding these expressions gives terms depending only on 
	$\nabla \nu_1$, $\nabla \nu_2$, $\nabla \nu_3$, $\nabla_X T_3$, and $\nabla_X JT_3$.  
	Taking inner products with $T_3$ and $JT_3$, and equation $(b)$ together with 
	Lemma~\ref{lema:nuevo}, give the required expressions.
\end{proof}
\section{Bonnet Mates}\label{subsec:compañeros de Bonnet}

Let $(\Sigma, \mathrm{d}s^{2})$ be an orientable Riemannian surface and let 
$f:\Sigma \to M$ be an isometric immersion.

Let $N$ be a unit normal vector field to $f$, and let $S$ be its associated 
shape operator. We choose a positively oriented frame $\{e_{1},e_{2}\}$ 
such that $S e_{i} = \kappa_{i} e_{i}$ for $i\in\{1,2\}$. This frame can be 
oriented with respect to the complex structure $J$ on $\Sigma$ induced by $N$, that is, $Je_{1}=e_{2}$ and $Je_{2}=-e_{1}$. Note that 
$\{e_{1},e_{2}\}$ is determined (up to sign) at points that are not umbilical.

\begin{remark}\label{rmk}
	To ensure that our frames are well defined, we work away from totally umbilical open subsets, which will be treated separately when needed. Notice that, totally umbilical surfaces in the spaces $\mathbb{E}(\kappa,\tau)$ and in $\mathrm{Sol}_{3}$ are already classified by Souam and Toubiana~\cite{SouamToubiana}. Namely, the only totally umbilical 
	surfaces in $\mathbb{E}(\kappa,0)=\mathbb{M}^{2}(\kappa)\times\mathbb{R}$ are 
	open subsets of horizontal slices $\mathbb{M}^{2}(\kappa)\times\{t_{0}\}$ for 
	some $t_{0}\in\mathbb{R}$, vertical planes $\gamma\times\mathbb{R}$, where 
	$\gamma$ is a geodesic of $\mathbb{M}^{2}(\kappa)$ or rotational surfaces with non-constant mean curvature described in~\cite[Thm. 9]{SouamToubiana}. When $\tau\neq 0$, no totally umbilical surfaces exist. 
	In $\mathrm{Sol}_{3}$, the only ones are the totally geodesic planes (integral 
	surfaces of either distribution 
	$\mathrm{span}\{E_{1}+E_{2},E_{3}\}$ or $\mathrm{span}\{E_{1}-E_{2},E_{3}\}$), 
	and a certain invariant surface described in 
	\cite[pp.\ 696--697]{SouamToubiana}. \hfill $\blacksquare$
\end{remark}

If $\widetilde{f}:\Sigma \to M$ is another isometric immersion of the same surface,
we denote its fundamental data using the same notation as for $f$, but with a tilde. Similarly, the immersion $\widetilde{f}$ induces an orientation $\widetilde{J}$ 
on $\Sigma$ and we may define a positively oriented frame 
$\{\widetilde{e}_{1},\widetilde{e}_{2}\}$ such that 
$\widetilde{S}\,\widetilde{e}_{i}=\widetilde{\kappa}_{i}\,\widetilde{e}_{i}$ 
for $i\in\{1,2\}$. There is no loss of generality in assuming 
$\kappa_{i}=\widetilde{\kappa}_{i}$ for $i\in\{1,2\}$. Otherwise, we may replace 
the frame $\{\widetilde{e}_{1},\widetilde{e}_{2}\}$ with the positively oriented  (with respect to $\widetilde{J}$) orthonormal frame $e'_{1}=\widetilde{e}_{2}$, $e'_{2}=-\widetilde{e}_{1}$ for which the corresponding principal curvatures are $\kappa'_{1}=\widetilde{\kappa}_{2}$ and $\kappa'_{2}=\widetilde{\kappa}_{1}$.

\begin{definition}
	We say that $f$ and $\widetilde{f}$ are \emph{Bonnet mates} if both immersions 
	share the same principal curvatures $\kappa_{1}$ and $\kappa_{2}$, as functions 
	on $\Sigma$, and there is no ambient isometry $\Phi\in\mathrm{Iso}(M)$ such that 
	$\Phi\circ f=\widetilde{f}$.
\end{definition}

Note that if two isometric immersions $f$ and $\widetilde{f}$ are Bonnet mates, 
then once a unit normal $N$ for $f$ is chosen (which fixes an orientation $J$ on 
$\Sigma$), there exists a unique unit normal $\widetilde{N}$ for $\widetilde{f}$ 
(fixing $\widetilde{J}$) away from minimal points. Following the terminology of 
Gálvez–Martínez–Mira~\cite{GMM}, we say that two immersions $f$ and 
$\widetilde{f}$ are \emph{positive Bonnet mates} if they induce the same 
orientation on $(\Sigma,\mathrm{d}s^{2})$, that is, if $J=\widetilde{J}$, and 
\emph{negative Bonnet mates} otherwise.
\begin{remark}
	When working in an open region where the surface is minimal, we may assume, without loss of generality, that the Bonnet mate is positive, since we are free to choose the orientation and both choices yield the same principal curvatures. \hfill $\blacksquare$
\end{remark}
By the definition of Bonnet mates, the eigenvalues of $S$ and $\widetilde{S}$ 
coincide. This implies that, for each point $p\in\Sigma$, there exists an 
isometry $I_{p}:T_{p}\Sigma\to T_{p}\Sigma$ such that $\widetilde{S}_{p}=I_{p}\circ S_{p}\circ I_{p}^{-1}$.
\begin{itemize}
	\item If $\widetilde{J}=J$ (positive Bonnet mates), then $I_{p}$ is a 
	rotation and there exists a function 
	$\theta:\Sigma\to\mathbb{R}\pmod{2\pi}$ such that $I=\mathrm{Rot}_{\theta}$. 
	In this case, the orthonormal frames are related by
	\begin{equation}\label{eq:frames positive orientation}
		\widetilde{e}_{1}=\cos\theta\, e_{1}+\sin\theta\, e_{2},\qquad
		\widetilde{e}_{2}=-\sin\theta\, e_{1}+\cos\theta\, e_{2}.
	\end{equation}
	
	\item If $\widetilde{J}=-J$ (negative Bonnet mates), then $I_{p}$ is an 
	axial symmetry and there exists a function 
	$\theta:\Sigma\to\mathbb{R}\pmod{2\pi}$ such that 
	$I=\mathrm{Sym}_{\theta}$, a field of axial reflections. By analogy with 
	\eqref{eq:frames positive orientation}, this means that the frames satisfy
	\begin{equation}\label{eq:frames negative orientation}
		\widetilde{e}_{1}=\cos\theta\, e_{1}+\sin\theta\, e_{2},\qquad
		\widetilde{e}_{2}=\sin\theta\, e_{1}-\cos\theta\, e_{2}.
	\end{equation}
\end{itemize}

These relations are well defined at non-umbilical points, and we omit explicitly 
the dependence on $p$ for simplicity. Since $\widetilde{S}$ is invariant under 
the change $\theta\mapsto \theta+2\pi$, we consider $\theta$ as a circle-valued 
function $\theta:\Sigma\to \mathbb{R}\pmod{2\pi}$. Observe that, if $\Sigma$ is 
simply connected, $\theta$ admits a well-defined global lift to a real 
function $\theta:\Sigma\to \mathbb{R}$. In general, such an extension need not 
exist for other topologies, so we 
work modulo $2\pi$.

\subsection{Isometries of invariant surfaces}\label{subsec:invariantes}
We now focus on surfaces invariant under a one-parameter group of isometries. As we will see, these surfaces admit Bonnet mates, which is due to the fact that they possess intrinsic isometries that do not (in general) arise from ambient isometries.

Let $K$ be a complete Killing field on $M$. This field generates a one-parameter subgroup $\{\varphi_t\}_{t\in\mathbb{R}}$ of $\mathrm{Iso}(M)$. We say that an isometric immersion $f:\Sigma\to M$ is \emph{invariant} under $K$ if $\varphi_t\circ f = f\circ\psi_t$ for some one-parameter group $\{\psi_t\}_{t\in\mathbb{R}}$ of isometries of $\Sigma$. The vector field associated to $\psi_t$ is $K^\top$ via $f$ and is compatible on $\Sigma$. Notice that $f(\Sigma)$ is a union of orbits of the action of $\{\varphi_t\}$ and, in particular, $K^\top$ is complete on $\Sigma$.

Let $(\Sigma,\mathrm{d} s^2)$ be a surface invariant under a intrinsic one-parameter group of isometries generated by a complete Killing field $Z\in\mathfrak{X}(\Sigma)$, and let $\{\psi_t\}_{t\in\mathbb{R}}$ denote the corresponding one-parameter group. We highlight the following curves, which may be defined at each point of $\Sigma$ and which will be useful for describing the isometries of $\Sigma$:
\begin{itemize}
	\item the integral curves $\alpha_p$ defined by $\alpha_p(t)=\psi_t(p)$ with $\alpha_p(0)=p$;
	\item the orthogonal curves $\beta_p$ defined by $\beta_p'(t)=JZ(\beta_p(t))$ with $\beta_p(0)=p$.
\end{itemize}

\noindent We also define $w = -JZ/\|Z\|$ and consider the local orthonormal frame $\{w,Jw\}$ on $\Sigma$. This frame is globally defined on a dense subset of $\Sigma$, namely at points where $Z\neq 0$.

In general, a surface invariant under a one-parameter group of ambient isometries $\varphi_t$ also admits the 
intrinsic isometries $\psi_t:\Sigma\to\Sigma$ which corresponds to $Z = K^\top$. (Observe that the orthonormal frame $\{w, Jw\}$ is invariant under these isometries.)  
Moreover, for each fixed point $o\in\Sigma$, there exists an isometry 
$\phi_o:\Sigma\to\Sigma$ such that $\phi_o\circ\phi_o=\mathrm{Id}$, $\phi_o$ leaves the 
curve $\beta_o$ invariant, and for every $p\in\Sigma$ the differential $(d\phi_o)_p$ satisfies $(d\phi_o)_p(w_p)=w_{\phi(p)}$ and $(d\phi_o)_p(Jw_p)=-Jw_{\phi(p)}$. In general, $\phi_o$ is only defined locally when $\pi \circ \beta$ is periodic, and it corresponds to the transformation $(t,s) \mapsto (-t,s)$ in the local coordinates described in~\cite[p.~17]{LerMan}. These isometries $\phi_o$ may be expressed in terms of the flow $\varphi_t$ as
\begin{equation}\label{eq:phio en terminos de varphit}
	\phi_o(\varphi_t(\beta_o(s))) = \varphi_{-t}(\beta_o(s)),\qquad \text{for all }t,s.
\end{equation}
Note that if the Killing field were not complete, these isometries would be defined only locally.

\begin{remark}\label{obs:propiamente invariante}
	Let $(\Sigma, \mathrm{d} s^2)$ be a Riemannian surface and let $f:\Sigma\to M$ be an isometric immersion invariant under a one-parameter group of isometries induced by a complete Killing field $K\in\mathfrak{X}(M)$ in the ambient space. In this setting, we define locally (where $\phi_o$ makes sense) the isometric immersion $\widetilde{f}=f\circ\phi_o:\Sigma\to M$ for some $o\in\Sigma$, which is congruent to $f$ if and only if there exists $\varphi\in\mathrm{Iso}(M)$ such that $\varphi\circ f = f\circ\phi_o$. Locally and via $f$, this is equivalent to requiring that $\phi_o$ be the restriction of some ambient isometry. From the description of $\phi_o$, we deduce that it arises from an ambient isometry precisely when the fibers of the Killing field $Z$ remain invariant and simultaneously have their orientation reversed. In this case, the isometry $\Phi:M\to M$ must satisfy $\Phi^2 = \mathrm{Id}$. Since $\phi_o(\beta_o(s)) = \beta_o(s)$, we have two possibilities: either $\beta_o$ is a geodesic in the ambient space and $\Phi$ is an axial symmetry with respect to it; or $\beta_o$ lies inside a totally geodesic surface of the ambient space and $\Phi$ is a mirror symmetry with respect to that surface. \hfill $\blacksquare$
\end{remark}

The previous observation motivates the following definition, since in order to construct Bonnet mates of this type we will need to restrict the class of invariant immersions under consideration.

\begin{definition}\label{def}
	We say that an isometric immersion $f:\Sigma\to M$ invariant under a one-parameter group of isometries is \emph{properly invariant} if there is no isometry of $M$ that leaves invariant each of the Killing field orbits contained in the surface and that simultaneously reverses the orientation of these fibers.
\end{definition}

\section{The Bonnet problem in Bianchi-Cartan-Vr\u{a}nceanu spaces}
Let $(\Sigma, \mathrm{d} s^2)$ be an orientable Riemannian surface and let 
$f:\Sigma \rightarrow \mathbb{E}(\kappa,\tau)$ be an isometric immersion whose fundamental data are 
given by $(J, S, T_3, \nu_3)$. Given another isometric immersion 
$\widetilde{f}:\Sigma \rightarrow \mathbb{E}(\kappa,\tau)$, we denote its 
fundamental data by $(\widetilde{J}, \widetilde{S}, \widetilde{T}_3, \widetilde{\nu}_3)$. 
From now on, we assume that both immersions are Bonnet mates.

Since the Gaussian curvature $K$ is an intrinsic property of $\Sigma$, using the Gauss equation in 
Proposition~\ref{prop:ecuaciones de compatibilidad ekt} for two Bonnet mates 
$f$ and $\widetilde{f}$, we obtain
\[
\det(S)+\tau^2 + (\kappa-4\tau^2)\nu_3^2 = 
K =  
\det(\widetilde{S})+\tau^2 + (\kappa-4\tau^2)\widetilde{\nu}_3^2.
\]
Because $S$ and $\widetilde{S}$ have the same determinant (their principal curvatures coincide), 
we conclude that $\widetilde{\nu}_3^2 = \nu_3^2$.

Moreover, from identity \textsc{(iv)} in Lemma~\ref{th:adapted theorem}, we deduce that
\begin{equation}\label{eq:equal norm T}
	\langle T_3, T_3 \rangle =  1 - \nu_3^2 = 
	\langle \widetilde{T}_3, \widetilde{T}_3 \rangle,
\end{equation}
so there exists a function $\psi: \Sigma \rightarrow \mathbb{R} \pmod {2\pi}$ such that 
$\widetilde{T}_3 = \mathrm{Rot}_\psi\, T_3$. We express the gradients of the angle functions in terms of the corresponding orthonormal frames 
$\{e_1, e_2\}$ and $\{\widetilde{e}_1, \widetilde{e}_2\}$ as
\begin{equation}\label{eq:frame expression for nabla nu}
	\begin{aligned}
		\nabla\nu_3 &= (-\kappa_1 \langle T_3, e_1\rangle 
		+ \tau \langle T_3, e_2\rangle)\, e_1 
		- (\kappa_2 \langle T_3, e_2\rangle 
		+ \tau \langle T_3, e_1\rangle)\, e_2,\\
		\nabla \widetilde{\nu}_3 &= (-\kappa_1 \langle\widetilde{T}_3, e_1\rangle 
		+ \tau \langle \widetilde{T}_3, \widetilde{e}_2\rangle)\, \widetilde{e}_1 
		- (\kappa_2 \langle \widetilde{T}_3, \widetilde{e}_2\rangle 
		+ \tau \langle \widetilde{T}_3, \widetilde{e}_1\rangle)\, \widetilde{e}_2.
	\end{aligned}
\end{equation}
\subsection{The Bonnet problem in product spaces $\mathbb{E}(\kappa,0)$}\label{subsec:bonnet mates product case}

Since we are in the case $\tau=0$, by item~(1) in Remark~\ref{remark:isometries and fundamental data}, 
we assume, without loss of generality, that $\widetilde{J}=J$.  So we fix the structure $(\Sigma, \mathrm{d} s^2, J)$ and assume that the Bonnet mates are positive. With this, apart from umbilical points, there exists a function 
$\theta: \Sigma \rightarrow \mathbb{R} \pmod {2\pi}$ such that $\{e_1, e_2\}$ and $\{\widetilde{e}_1, \widetilde{e}_2\}$ are related by~\eqref{eq:frames positive orientation}. Moreover, since $\widetilde{\nu}_3^2 = \nu_3^2$, using item~(3) in 
Remark~\ref{remark:isometries and fundamental data} (this item preserves $\nu_3$, so item~(1) is compatible with it), we additionally assume 
that $\widetilde{\nu}_3 = \nu_3$.

In the following result we work with surfaces that are neither totally umbilical nor minimal. 
The former will be discussed in Remark~\ref{remark:Bonnet mates totally umbilic}; regarding the latter, 
recall that every minimal surface admits an $\mathbb{S}^1$-family of Bonnet mates, the so-called 
\emph{associate family}, described in~\cite{Dan09, HST}. The uniqueness of this family was established 
by Daniel~\cite[Thm.~2.4]{Daniel15}, who proved that two locally isometric minimal surfaces sharing the 
same angle function $\nu_3$ must belong to the same associate family.

We will work with the following sets:
\begin{equation}\label{sets}
\begin{aligned}
	M_1 &= \{p\in \Sigma : \|\nabla\nu_3\|^2 \neq 0\},\\
	M_2 &= \{p\in \Sigma : \|\nabla\nu_3\|^2 = 0,\ \|\nabla H\|^2 \neq 0\},\\
	M_3 &= \{p\in \Sigma : \|\nabla\nu_3\|^2 = 0,\ \|\nabla H\|^2 = 0\}.
\end{aligned}
\end{equation}
These three sets are pairwise disjoint and satisfy $\Sigma = M_1 \cup M_2 \cup M_3$. By a simple topological argument, at least one of them must have nonempty interior. Moreover, under the hypothesis of real analyticity, one of these three sets is dense in $\Sigma$, that is, $\Sigma = \overline{M}_i$ for some $i\in\{1,2,3\}$. The surfaces satisfying condition $M_3$ were classified by Espinar and Rosenberg~\cite{EspinarRosenberg11}. However, as shown by Domínguez-Vázquez and Manzano in~\cite{DM}, these surfaces agree with those having constant principal curvatures, as well as with extrinsically homogeneous surfaces, which will be important for our purposes.

In the next lemma we study an upper bound for the number of Bonnet mates that an immersion in 
$\mathbb{M}^2(\kappa)\times\mathbb{R}$ admits. Observe that determining the possible values of the function 
$\theta$, as well as of the scalar products 
$\langle \widetilde{T}_3, \widetilde{e}_1\rangle$ and 
$\langle \widetilde{T}_3, \widetilde{e}_2\rangle$ in terms of the fundamental data 
$(J,S,T_3,\nu_3)$, yields the uniqueness of the Bonnet mate, as these quantities determine the remaining 
fundamental data for $\widetilde{f}$. More specifically, we have $\widetilde{S} = \mathrm{Rot}_\theta \circ S \circ \mathrm{Rot}_{-\theta}$ and $\widetilde{T}_3 
=  \langle \widetilde{T}_3, \widetilde{e}_1\rangle\, \mathrm{Rot}_\theta e_1 
+ \langle \widetilde{T}_3, \widetilde{e}_2\rangle\, \mathrm{Rot}_\theta e_2$.

\begin{lemma}\label{lemma:bonnet mate cases tau 0}
	Let $(\Sigma, \mathrm{d} s^2, J)$ be an oriented Riemannian surface and let 
	$f:\Sigma \rightarrow \mathbb{M}^2(\kappa)\times\mathbb{R}$ be an isometric real-analytic immersion 
	that is neither minimal nor totally umbilical 
	(see Remark~\ref{remark:Bonnet mates totally umbilic}). 
	If $f$ admits a Bonnet mate $\widetilde{f}$, then one of the following possibilities occurs:
	\begin{itemize}
		\item[(a)] If $\Sigma = \overline{M}_1$, then the relations 
		$\langle \widetilde{T}_3, \widetilde{e}_1\rangle = \langle T_3, e_1\rangle$ and 
		$\langle \widetilde{T}_3, \widetilde{e}_2\rangle = -\langle T_3, e_2\rangle$ hold (up to ambient isometries), 
		and $\theta$ is uniquely determined (modulo $2\pi$) by
		\begin{equation}\label{eq:angle theta tau 0}
			\cos\theta = \frac{\kappa_1^2 \langle T_3, e_1\rangle^2 - \kappa_2^2 \langle T_3, e_2\rangle^2}
			{\kappa_1^2 \langle T_3, e_1\rangle^2 + \kappa_2^2 \langle T_3, e_2\rangle^2}, 
			\qquad
			\sin\theta = -\frac{2 \kappa_1 \kappa_2 \langle T_3, e_1\rangle \langle T_3, e_2\rangle}
			{\kappa_1^2 \langle T_3, e_1\rangle^2 + \kappa_2^2 \langle T_3, e_2\rangle^2}.
		\end{equation}
		Consequently, $f$ has at most one Bonnet mate.
		
		\item[(b)] If $\Sigma = \overline{M}_2$, then 
		$\langle \widetilde{T}_3, \widetilde{e}_1\rangle = \langle T_3, e_1\rangle$ and 
		$\langle \widetilde{T}_3, \widetilde{e}_2\rangle = \langle T_3, e_2\rangle = 0$ 
		hold (up to ambient isometries), and $\theta$ is uniquely determined (modulo $2\pi$) by
		\begin{equation}\label{eq:angle theta tau 0 nabla nu 0}
			\cos\theta = 
			\frac{\langle \nabla H, T_3\rangle^2 - \langle \nabla H, JT_3\rangle^2}
			{\langle \nabla H, T_3\rangle^2 + \langle \nabla H, JT_3\rangle^2},
			\qquad
			\sin\theta =
			\frac{2 \langle \nabla H, T_3\rangle \langle \nabla H, JT_3\rangle}
			{\langle \nabla H, T_3\rangle^2 + \langle \nabla H, JT_3\rangle^2}.
		\end{equation}
		As a result, $f$ has at most one Bonnet mate.
		
		\item[(c)] If $\Sigma = \overline{M}_3$, then the immersions have constant principal curvatures and there exists, at most, a one-parameter family of Bonnet mates parametrized by $\mathbb{S}^1$.
	\end{itemize}
\end{lemma}
\begin{proof}
	Since $\nu_3 = \widetilde{\nu}_3$, we can expand the equality 
	$||\nabla \nu_3||^2 = ||\nabla \widetilde{\nu}_3||^2$ as follows:
	\[
	\kappa_1^2 \langle T_3, e_1\rangle^2 + \kappa_2^2 \langle T_3, e_2\rangle^2 
	= \kappa_1^2 \langle \widetilde{T}_3, \widetilde{e}_1\rangle^2 
	+ \kappa_2^2 \langle \widetilde{T}_3, \widetilde{e}_2\rangle^2.
	\]
	Using identity~\eqref{eq:equal norm T}, written in terms of the orthonormal frames 
	$\{e_1, e_2\}$ and $\{\widetilde{e}_1, \widetilde{e}_2\}$, we also have
	\[
	\langle T_3, e_1\rangle^2 + \langle T_3, e_2\rangle^2
	= \langle \widetilde{T}_3, \widetilde{e}_1\rangle^2 
	+  \langle \widetilde{T}_3, \widetilde{e}_2\rangle^2.
	\]
	Combining both expressions and setting 
	$x_i = \langle T_3, e_i\rangle^2 - \langle \widetilde{T}_3, \widetilde{e}_i\rangle^2$ 
	for $i\in\{1,2\}$, we obtain the following homogeneous linear system: $\{\kappa_1^2 x_1 + \kappa_2^2 x_2 = 0$, $x_1 + x_2 = 0\}$ in $x_1$ and $x_2$, whose coefficient matrix has 
	determinant $\kappa_1^2 - \kappa_2^2$.
	
	Since we work with surfaces that are neither totally umbilical $(\kappa_1 = \kappa_2)$ nor minimal $(\kappa_1 = - \kappa_2)$, we can restrict ourselves (by real analyticity) to an open set where $\kappa_1^2 - \kappa_2^2 \neq 0$, 
	and hence $x_1 = x_2 = 0$, that is,
	\begin{equation}\label{eq:xi = widetildexi}
		\langle T_3, e_1\rangle^2 = \langle \widetilde{T}_3, \widetilde{e}_1\rangle^2, 
		\qquad 
		\langle T_3, e_2\rangle^2 = \langle \widetilde{T}_3, \widetilde{e}_2\rangle^2.
	\end{equation}
	On the other hand, rewriting $\nabla \nu_3 = \nabla \widetilde{\nu}_3$ as in 
	\eqref{eq:frame expression for nabla nu} and substituting $e_1$ and $e_2$ from~\eqref{eq:frames positive orientation}, we obtain
	\begin{align*}
		\kappa_1 \langle \widetilde{T}_3, \widetilde{e}_1\rangle \widetilde{e}_1 
		+ \kappa_2 \langle \widetilde{T}_3, \widetilde{e}_2\rangle \widetilde{e}_2 
		&= 
		\left(\kappa_1 \langle T_3, e_1\rangle \cos\theta 
		- \kappa_2 \langle T_3, e_2\rangle \sin\theta\right)\widetilde{e}_1 + 
		\left(\kappa_1 \langle T_3, e_1\rangle \sin\theta 
		+ \kappa_2  \langle T_3, e_2\rangle \cos\theta\right)\widetilde{e}_2.
	\end{align*}
	Equating coefficients in the basis $\{\widetilde{e}_1, \widetilde{e}_2\}$ gives the following linear system in the unknowns $\cos\theta$ and $\sin\theta$:
	\begin{equation}\label{eq:equation theta nabla nu3 non constant} 
		\begin{aligned}
			\kappa_1 \langle \widetilde{T}_3, \widetilde{e}_1\rangle 
			&= \kappa_1 \langle T_3, e_1\rangle \cos\theta 
			- \kappa_2 \langle T_3, e_2\rangle \sin\theta,\\
			\kappa_2 \langle \widetilde{T}_3, \widetilde{e}_2\rangle 
			&= \kappa_2 \langle T_3, e_2\rangle \cos\theta 
			+ \kappa_1 \langle T_3, e_1\rangle \sin\theta.
		\end{aligned}
	\end{equation}
	This system has a unique solution when $\kappa_1^2 \langle T_3, e_1\rangle^2 + \kappa_2^2 \langle T_3, e_2\rangle^2 \neq 0$, i.e., when $||\nabla\nu_3||^2 \neq 0$.
	
	As noted before the lemma, at least one of $M_1$, $M_2$, or $M_3$ is dense in $\Sigma$. 
	We now analyze each case separately.
	
	Assume first $\Sigma = \overline{M}_1$, that is, $||\nabla\nu_3||^2 \neq 0$.  
	Then, \eqref{eq:equation theta nabla nu3 non constant} has a unique solution 
	for each choice of sign in~\eqref{eq:xi = widetildexi}: there are four possibilities, 
	depending on the signs chosen.  
	If $\langle \widetilde{T}_3, \widetilde{e}_i\rangle = \langle T_3, e_i\rangle$ or 
	$\langle \widetilde{T}_3, \widetilde{e}_i\rangle = -\langle T_3, e_i\rangle$ for 
	$i \in \{1,2\}$, then the solutions are $\theta = 0$ or $\theta = \pi$, respectively.  
	This corresponds to $\widetilde{S} = S$, $\widetilde{T}_3 = T_3$ or 
	$\widetilde{S} = -S$, $\widetilde{T}_3 = -T_3$.  
	In both cases, $f$ and $\widetilde{f}$ are congruent, as they represent an immersion and its mirror-symmetric image, and the fundamental data determine them uniquely.
	
	Now assume 
	$\langle \widetilde{T}_3, \widetilde{e}_1\rangle = \langle T_3, e_1\rangle$ and 
	$\langle \widetilde{T}_3, \widetilde{e}_2\rangle = -\langle T_3, e_2\rangle$.  
	Then the angle function $\theta$ is given by~\eqref{eq:angle theta tau 0}.  
	If instead 
	$\langle \widetilde{T}_3, \widetilde{e}_1\rangle = -\langle T_3, e_1\rangle$ and 
	$\langle \widetilde{T}_3, \widetilde{e}_2\rangle = \langle T_3, e_2\rangle$, we obtain
	\[
	\cos\theta = 
	\frac{\kappa_1^2 \langle T_3, e_1\rangle^2 - \kappa_2^2 \langle T_3, e_2\rangle^2}
	{\kappa_1^2 \langle T_3, e_1\rangle^2 + \kappa_2^2 \langle T_3, e_2\rangle^2},
	\qquad
	\sin\theta = 
	-\frac{2 \kappa_1 \kappa_2 \langle T_3, e_1\rangle \langle T_3, e_2\rangle}
	{\kappa_1^2 \langle T_3, e_1\rangle^2 + \kappa_2^2 \langle T_3, e_2\rangle^2}.
	\]
	Note that both cases yield the same function $\theta$ up to a shift of $\pi$, which means 
	$\widetilde{S}$ is the opposite of the shape operator given by~\eqref{eq:angle theta tau 0}, 
	and this case is also congruent by 
	Remark~\ref{remark:isometries and fundamental data}.  
	Thus, we may assume 
	$\langle \widetilde{T}_3, \widetilde{e}_1\rangle = \langle T_3, e_1\rangle$ and 
	$\langle \widetilde{T}_3, \widetilde{e}_2\rangle = -\langle T_3, e_2\rangle$ (up to ambient 
	isometries), which corresponds to case $(a)$ in the lemma.
	
	Now assume that the complement $\Sigma \setminus M_1 = M_2 \cup M_3$ has non-empty interior, i.e., $||\nabla \nu_3||^2 = 0$ everywhere, so that~\eqref{eq:equation theta nabla nu3 non constant} no longer has a unique solution. By compatibility equation \textsc{(ii)} in Lemma~\ref{th:adapted theorem}, we have 
	$ST_3=0$ and $\widetilde{S}\widetilde{T}_3=0$.  
	Since the only surfaces with $\nu_3^2 = 1$ are totally umbilical, we restrict to open sets where $\nu_3^2<1$, so 
	$T_3,\widetilde{T}_3\neq 0$.  
	In this case, $T_3$ and $\widetilde{T}_3$ are eigenvectors of $S$ and $\widetilde{S}$, 
	and one principal curvature vanishes. It immediately follows that
	\[
	e_1 = \frac{T_3}{||T_3||}, 
	\quad e_2 = \frac{JT_3}{||T_3||}, \quad 
	\widetilde{e}_1 = \frac{\widetilde{T}_3}{||T_3||}, \quad 
	\widetilde{e}_2 = \frac{J\widetilde{T}_3}{||T_3||},
	\]
	with principal curvatures $\kappa_1=0$ and $\kappa_2=2H$.  
	So we have 
	$\langle \widetilde{T}_3, \widetilde{e}_1\rangle 
	= \langle T_3, e_1\rangle = \sqrt{1-\nu_3^2}$ 
	and 
	$\langle \widetilde{T}_3, \widetilde{e}_2\rangle 
	= \langle T_3, e_2\rangle = 0$.  
	Thus, $\widetilde{T}_3 
	= \langle \widetilde{T}_3, \widetilde{e}_1\rangle \widetilde{e}_1 
	= \langle T_3, e_1\rangle \mathrm{Rot}_\theta e_1 
	= \mathrm{Rot}_\theta T_3$ and therefore $\theta = \psi$. We now consider the two remaining cases.
	
	Assume $\Sigma = \overline{M}_2$, i.e., 
	$||\nabla \nu_3||^2 = 0$ and $||\nabla H||^2 \neq 0$ on a dense set.  
	For equation~\textsc{(iii)} in Lemma~\ref{th:adapted theorem} to hold, we must have
	$\langle \nabla H, T_3\rangle = \langle \nabla H, \widetilde{T}_3\rangle$, which rewrites as
	\[
	\langle \nabla H, T_3\rangle 
	= \langle \nabla H, T_3\rangle \cos\theta 
	+ \langle \nabla H, JT_3\rangle \sin\theta.
	\]
	This equation has two solutions in $[0,2\pi)$: $\theta = 0$, corresponding to congruent 
	immersions, and the angle determined by~\eqref{eq:angle theta tau 0 nabla nu 0}.  
	Thus, even in this case $(b)$, if it exists, the Bonnet mate is uniquely determined.
	
	Finally, assume $\Sigma = \overline{M}_3$. Here $||\nabla\nu_3||^2 = 0$ and 
	$||\nabla H||^2 = 0$, and as mentioned, the surface has constant principal curvatures 
	\cite{DM}. Under these hypotheses, the identities 
	$\nabla \nu_3 = \nabla \widetilde{\nu}_3$ and 
	$\Delta \nu_3 = \Delta \widetilde{\nu}_3$ hold automatically by 
	Lemma~\ref{th:adapted theorem}.  
	Thus, any difference between the immersions must come from equation \textsc{(i)} of the 
	same lemma.  
	This equation holds if and only if 
	$\nabla_X \widetilde{T}_3 
	= \langle \nabla\theta, X\rangle J\widetilde{T}_3 
	+ \mathrm{Rot}_\theta\nabla_X T_3$.  
	Taking the inner product with $J\widetilde{T}_3$ and using that $T_3$ and $\widetilde{T}_3$ 
	satisfy the compatibility equations, we obtain
	\[
	(1-\nu_3^2) \langle \nabla\theta, X\rangle 
	= \nu_3 \langle X, 2H(JT_3 - J\widetilde{T}_3)\rangle.
	\]
	Since we work where $\nu_3^2 \neq 1$, we get
\[\nabla \theta 
		= \frac{2H\nu_3}{1-\nu_3^2} J(T_3 - \mathrm{Rot}_\theta T_3).\]
	Taking a local chart, this becomes a first-order differential equation for $\theta$. So, for each initial value $\theta_0 \in \mathbb{R}\pmod{2\pi}$, there is at most one solution. This determines the fundamental data of the immersion and, in this case, shows that there is at most a one-parameter family of Bonnet mates.
\end{proof}

\begin{remark}\label{remark:Bonnet mates totally umbilic}
It is also natural to ask whether totally umbilical immersions may admit Bonnet mates where it is not well defined the function $\theta$ (see Section~\ref{subsec:compañeros de Bonnet}). Firstly, horizontal slices admit no Bonnet mates, any isometry of $\mathbb{M}^2(\kappa)\simeq \mathbb{M}^2(\kappa)\times\{t_0\}$ 
extends to an isometry of $\mathbb{M}^2(\kappa)\times\mathbb{R}$ 
by~\cite[Lem.~2.10]{Man14}.  

Vertical planes have constant principal curvatures, so by the same argument as in 
Lemma~\ref{lemma:bonnet mate cases tau 0} (using $\psi$ instead of $\theta$) they admit, 
at most, an $\mathbb{S}^1$-family of Bonnet mates.  
Since these planes are intrinsically flat, one may locally apply an internal rotation 
that is not an ambient isometry, producing non-congruent immersions with the same 
principal curvatures, a fact we will use later.

If $\nabla H\neq 0$, the equality $\Delta\widetilde{\nu}_3=\Delta\nu_3$ implies, as in 
Lemma~\ref{lemma:bonnet mate cases tau 0}, $\langle \nabla H, T_3\rangle 
= \langle \nabla H, T_3\rangle\cos\psi 
+ \langle \nabla H, JT_3\rangle\sin\psi$, which determines $\psi$ uniquely.  
Hence $\widetilde{S}=S$ and $\widetilde{T}_3=\mathrm{Rot}_\psi T_3$ are fixed, so at most one Bonnet mate exists. \hfill $\blacksquare$
\end{remark}
This result lets us complete the classification of the Bonnet problem in $\mathbb{M}^2(\kappa)\times\mathbb{R}$ as formulated in~\cite{GMM}.
\begin{theorem}\label{th:main theorem tau 0}
	Let $(\Sigma, \mathrm{d} s^2, J)$ be an oriented real-analytic Riemannian surface, and let 
	$f:\Sigma \to \mathbb{M}^2(\kappa)\times \mathbb{R}$ be an isometric immersion that does not lie in a horizontal plane. Then $f$ admits a Bonnet mate if and only if one of the following conditions holds:
	\begin{itemize}
		\item the immersion $f$ is minimal,
		\item the immersion $f$ is properly invariant under a one-parameter group of isometries.
	\end{itemize}
	If $f$ is minimal or has constant principal curvatures, then there exists exactly a continuous $\mathbb{S}^1$-family of Bonnet mates. Otherwise, the Bonnet mate is unique.
\end{theorem}
\begin{proof}
		If the surface is minimal, the statement follows from the work of Daniel~\cite{Daniel15}. Otherwise, the immersion must be properly invariant under a one-parameter group of isometries and such invariant surfaces are known to admit Bonnet mates by~\cite{GMM}.
		
		If the immersion admits a Bonnet mate and satisfies $||\nabla \nu_3||^2 \neq 0$ or $||\nabla H||^2 \neq 0$ on a dense set, then, by Lemma~\ref{lemma:bonnet mate cases tau 0}, this Bonnet mate is unique. Otherwise, again by Lemma~\ref{lemma:bonnet mate cases tau 0}, if the surface satisfies both $\nabla \nu_3 = 0$ and $\nabla H = 0$, there exists at most one Bonnet mate for each $\theta_0 \in \mathbb{R} \ (\mathrm{mod}\ 2\pi)$. We prove that such mates in fact exist.
		
		These surfaces have constant $\nu_3$ and $H$, by~\cite{DM} this is equivalent to being extrinsically homogeneous.  
		Thus, given a point $p$ on the surface, we can take a neighborhood $U$ that is invariant under intrinsic rotations $\phi_{\theta_0}$ fixing $p$. Then the immersions $f : U \to \mathbb{E}(\kappa,0)$ and 
		$f \circ \phi_{\theta_0} : U \to \mathbb{E}(\kappa,0)$ are not pointwise congruent (we have excluded horizontal sections, where this fails; see Remark~\ref{remark:Bonnet mates totally umbilic}), but they share the same principal curvatures, since these are constant. Nevertheless, their images $f(U)$ and $f\big(\phi_{\theta_0}(U)\big)$ are congruent.		
\end{proof}
Having established the uniqueness of Bonnet mates, we have solved the version of the problem in $\mathbb{M}^2(\kappa)\times \mathbb{R}$ originally considered by Chern in~\cite{Chern}.
\begin{corollary}\label{cortau0}
	The only immersions in $\mathbb{M}^2(\kappa)\times \mathbb{R}$ that admit a continuous isometric deformation (not coming from ambient isometries) preserving the principal curvatures are the minimal surfaces and the surfaces with constant principal curvatures, excluding the horizontal slices $\mathbb{M}^2(\kappa)\times\{t_0\}$ for $t_0\in\mathbb{R}$.
\end{corollary}

\subsection{The Bonnet problem in $\mathbb{E}(\kappa,\tau)$ spaces}
In this subsection, we fix $\tau \neq 0$. Given an orientable Riemannian surface $(\Sigma, \mathrm{d} s^2)$, we assume the existence of two isometric immersions $f$, $\widetilde{f}:\Sigma \rightarrow \mathbb{E}(\kappa,\tau)$ that are Bonnet mates with corresponding induced orientations $J$ and $\widetilde{J}$ respectively. By item~(2) in Remark~\ref{remark:isometries and fundamental data}, we assume that $\widetilde{\nu}_3 = \nu_3$. In this case we must distinguish between $\widetilde{J}=J$ (positive Bonnet mates) and $\widetilde{J}=-J$ (negative Bonnet mates), since they behave differently. This distinction ultimately comes from the fact that no orientation-reversing ambient isometries exist when $\tau\neq 0$.

We begin by proving a lemma that will be used in both the positive and the negative cases, and then we treat each case separately. We consider the sets $M_1$, $M_2$, and $M_3$ introduced in~\eqref{sets}. Recall that, when $\tau \neq 0$, there are no totally umbilical surfaces in these spaces~\cite{SouamToubiana}.

\begin{lemma}\label{lemma:first cases}
	Let $(\Sigma, \mathrm{d} s^2)$ be an orientable Riemannian surface and let $f:\Sigma \rightarrow \mathbb{E}(\kappa,\tau)$ and $\widetilde{f}:\Sigma \rightarrow \mathbb{E}(\kappa,\tau)$ be Bonnet mates. Then the following relation holds
	\begin{align*}
		&\langle \widetilde{T}_3, \widetilde{e}_1\rangle = \cos\alpha \langle T_3, e_1\rangle + \sin\alpha \langle T_3, e_2\rangle,\\
		&\langle \widetilde{T}_3, \widetilde{e}_2\rangle = \sin\alpha \langle T_3, e_1\rangle - \cos\alpha \langle T_3, e_2\rangle,
	\end{align*}
	\noindent where $\alpha$ is one of the following (note that $H$ is not assumed to be constant a priori):
	\begin{align*}
		&\text{\textit{(1)}}\quad \alpha = 0,&  &\text{\textit{(2)}}\quad \alpha = \pi,\\
		&\text{\textit{(3)}}\quad \alpha = \operatorname{arctan}(-\tau/H) + \pi,& &\text{\textit{(4)}}\quad \alpha = \operatorname{arctan}(-\tau/H).
	\end{align*}
\end{lemma}
\begin{proof}
	Since $\widetilde{\nu}_3=\nu_3$, we expand the identity 
	$||\nabla\widetilde{\nu}_3||^2 = ||\nabla\nu_3||^2$ and use 
	equation~\eqref{eq:norm nabla nu}, written in the frames 
	$\{e_1,e_2\}$ and $\{\widetilde{e}_1,\widetilde{e}_2\}$, to obtain
	\begin{align*}
		&\kappa_1^2 \langle T_3, e_1\rangle^2 
		+ \kappa_2^2 \langle T_3, e_2\rangle^2 
		+ \tau^2(\langle T_3,e_1\rangle^2+\langle T_3,e_2\rangle^2)
		+ 2\tau(\kappa_2-\kappa_1)\langle T_3,e_1\rangle\langle T_3,e_2\rangle \\
		&= \kappa_1^2 \langle \widetilde{T}_3, \widetilde{e}_1\rangle^2 
		+ \kappa_2^2 \langle \widetilde{T}_3, \widetilde{e}_2\rangle^2 
		+ \tau^2(\langle \widetilde{T}_3,\widetilde{e}_1\rangle^2
		+ \langle \widetilde{T}_3,\widetilde{e}_2\rangle^2)
		+ 2\tau(\kappa_2-\kappa_1)\langle \widetilde{T}_3,\widetilde{e}_1\rangle
		\langle \widetilde{T}_3,\widetilde{e}_2\rangle.
	\end{align*}
	\noindent Moreover, using identity~\eqref{eq:equal norm T}, we also have $\langle T_3,e_1\rangle^2+\langle T_3,e_2\rangle^2
	= \langle \widetilde{T}_3,\widetilde{e}_1\rangle^2
	+ \langle \widetilde{T}_3,\widetilde{e}_2\rangle^2$.
	
	Thus we obtain a system of two quadratic polynomial equations in the variables 
	$\langle \widetilde{T}_3,\widetilde{e}_1\rangle$ and 
	$\langle \widetilde{T}_3,\widetilde{e}_2\rangle$.  
	By Bézout’s theorem, the system has at most four solutions, provided the two polynomials 
	are linearly independent. This is easily checked when $\kappa_1\neq\kappa_2$, and since we 
	may restrict to an open set where this holds, uniqueness follows there.  
	The proposed solutions can be computed explicitly and, since the functions are analytic, extends to all of $\Sigma$.
\end{proof}

Although we will maintain the same general framework, we now divide the analysis into two separate subsections, studying the positive and negative Bonnet problems separately.

\subsubsection{Uniqueness of the Bonnet problem: positive case}
Suppose that $f$ and $\widetilde{f}$ are positive Bonnet mates, that is, in addition to $\widetilde{\nu}_3 = \nu_3$, we can assume $\widetilde{J} = J$. Consequently, there exists an angle function $\theta:\Sigma \to \mathbb{R} \text{ (mod }2\pi)$ and a field of isometric endomorphisms $\mathrm{Rot}_\theta$ such that $\mathrm{Rot}_\theta (e_i) = \widetilde{e}_i$ for $i\in \{1,2\}$.

Positive Bonnet mates are necessarily restricted to the class of constant mean curvature surfaces as it is proved in~\cite{GMM}. When restricted to this class we have that $M_2 = \emptyset$ so $\Sigma = M_1 \cup M_3$.

\begin{lemma}\label{lemma:bonnet mate cases tau neq 0}
	Let $(\Sigma, \mathrm{d} s^2, J)$ be an oriented Riemannian surface and let $f:\Sigma \rightarrow \mathbb{E}(\kappa,\tau)$ be an isometric immersion. If $f$ admits a positive Bonnet mate $\widetilde{f}$, then the immersions have constant mean curvature and either:
	\begin{itemize}
		\item[(a)] $\Sigma = \overline{M}_1$, the relation \textit{(3)} in Lemma~\ref{lemma:first cases} holds and the function $\theta$ is uniquely determined by
		\begin{align}
			\cos\theta &= \tfrac{- \langle T_3, e_1\rangle^2\left(2H \kappa_1^2 + \tau^2(3\kappa_1 - \kappa_2)\right)+ 2\tau \langle T_3, e_1\rangle \langle T_3, e_2\rangle\left(\kappa_1^2 + \kappa_2^2 + 2\tau^2\right)+ \langle T_3, e_2\rangle^2\left(2H \kappa_2^2 - \tau^2(\kappa_1 - 3\kappa_2)\right)}{\sqrt{4 H^2 + 4\tau^2} \left(\langle T_3, e_1\rangle^2(\kappa_1^2 + \tau^2)+ 2\tau \langle T_3, e_1\rangle \langle T_3, e_2\rangle(\kappa_2 - \kappa_1)+ \langle T_3, e_2\rangle^2(\kappa_2^2 + \tau^2)\right)},\notag \\
			\sin\theta &= -\tfrac{2(\kappa_1 \kappa_2 + \tau^2) \left(\langle T_3, e_1\rangle \langle T_3, e_2\rangle(\kappa_1 + \kappa_2) + \tau(\langle T_3, e_1\rangle - \langle T_3, e_2\rangle)(\langle T_3, e_1\rangle + \langle T_3, e_2\rangle)\right)}{\sqrt{4 H^2 + 4\tau^2} \left(\langle T_3, e_1\rangle^2(\kappa_1^2 + \tau^2)+ 2\tau \langle T_3, e_1\rangle \langle T_3, e_2\rangle(\kappa_2 - \kappa_1)+ \langle T_3, e_2\rangle^2(\kappa_2^2 + \tau^2)
				\right)}, \label{eq:angle positive nabla nu neq 0}
		\end{align}
		in which case there is at most one positive Bonnet mate;
		\item[(b)] $\Sigma = \overline{M}_3$, the immersion has constant principal curvatures, in which case, there exists at most a one-parameter family of positive Bonnet mates parameterized by $\mathbb{S}^1$.
	\end{itemize}
\end{lemma}
\begin{proof}
	Equalizing the components in the identity $\nabla\widetilde{\nu}_3=\nabla\nu_3$, written as in~\eqref{eq:frame expression for nabla nu} and expressed in the frame $\{\widetilde{e}_1,\widetilde{e}_2\}$, yields the linear system
	\begin{align*}
		- \kappa_1 \langle T_3, e_1\rangle + \tau \langle T_3, e_2\rangle 
		&= (- \kappa_1 \langle \widetilde{T}_3, \widetilde{e}_1\rangle + \tau \langle \widetilde{T}_3, \widetilde{e}_2\rangle)\cos\theta 
		+ (-\kappa_2 \langle \widetilde{T}_3, \widetilde{e}_2\rangle - \tau \langle \widetilde{T}_3, \widetilde{e}_1\rangle)\sin\theta, \\[0.3em]
		- \kappa_2 \langle T_3, e_2\rangle - \tau \langle T_3, e_1\rangle 
		&= (- \kappa_2 \langle \widetilde{T}_3, \widetilde{e}_2\rangle - \tau \langle \widetilde{T}_3, \widetilde{e}_1\rangle)\cos\theta 
		- (-\kappa_1 \langle \widetilde{T}_3, \widetilde{e}_1\rangle + \tau \langle \widetilde{T}_3, \widetilde{e}_2\rangle)\sin\theta.
	\end{align*}
	This system has a unique solution for $\cos\theta$ and $\sin\theta$ if and only if 
	$||\nabla\nu_3||^2\neq 0$.
	
	Since we work in the real-analytic setting and $\Sigma = M_1 \cup M_3$, only occur either $\Sigma = \overline{M}_1$ or $\Sigma = M_3$.

	$(a)$ If $\Sigma = \overline{M}_1$, we restrict ourselves to an open set where $||\nabla\nu_3||^2\neq 0$, and the system admits 
	a unique solution in each of the four cases of Lemma~\ref{lemma:first cases}.  
	For the trivial solutions \textit{(1)} and \textit{(2)}, i.e.,
	$\langle\widetilde{T}_3,\widetilde{e}_i\rangle = \pm \langle T_3,e_i\rangle$ for all $i\in\{1,2\}$, we obtain 
	$\theta=0$ or $\theta=\pi$, hence $\widetilde{S}=S$ or $\widetilde{S}=-S$, and by 
	Remark~\ref{remark:isometries and fundamental data} no Bonnet mates arise.
	
	If instead we take case \textit{(3)} of Lemma~\ref{lemma:first cases}, namely
	\begin{align*}
		\langle \widetilde{T}_3, \widetilde{e}_1\rangle 
		&= \frac{-H \langle T_3, e_1\rangle + \tau \langle T_3, e_2\rangle}{\sqrt{H^2 + \tau^2}}, \qquad
		\langle \widetilde{T}_3, \widetilde{e}_2\rangle 
		= \frac{H \langle T_3, e_2\rangle + \tau \langle T_3, e_1\rangle}{\sqrt{H^2 + \tau^2}},
	\end{align*}
	then $\theta$ is determined by~\eqref{eq:angle positive nabla nu neq 0}.  
	Case \textit{(4)} yields the same function $\theta$ up to a shift of $\pi$, hence it is 
	congruent to the previous one.  
	Thus we obtain case~(a) of the statement, with $\widetilde{S}$ completely determined by 
	$\theta$, and
	\[
	\widetilde{T}_3
	= \langle \widetilde{T}_3, \widetilde{e}_1\rangle\, \widetilde{e}_1
	+ \langle \widetilde{T}_3, \widetilde{e}_2\rangle\, \widetilde{e}_2
	= \tfrac{-H \langle T_3, e_1\rangle + \tau \langle T_3, e_2\rangle}{\sqrt{H^2 + \tau^2}}\mathrm{Rot}_\theta e_1 
	+ \tfrac{H \langle T_3, e_2\rangle + \tau \langle T_3, e_1\rangle}{\sqrt{H^2 + \tau^2}}\mathrm{Rot}_\theta e_2.
	\]
	Hence the Bonnet mate is unique.
	
	$(b)$ If $\Sigma = \overline{M}^3$, i.e., $||\nabla\nu_3||^2 = 0$ and $||\nabla H||^2 = 0$, the surface has constant 
	principal curvatures by~\cite{DM}. As in Lemma~\ref{lemma:bonnet mate cases tau 0}, at most a one-parameter family of Bonnet mates can occur, completing the proof.
\end{proof}

With this in hand, we can now state a complete theorem for positive Bonnet mates in $\mathbb{E}(\kappa,\tau)$ in the simply connected case, which, in view of the previous lemma, is analogous to Theorem~\ref{th:main theorem tau 0}.

\begin{theorem}\label{teo1}
	Let $(\Sigma, \mathrm{d}s^{2}, J)$ be an oriented Riemannian surface, and let 
	$f:\Sigma \to \mathbb{E}(\kappa,\tau)$ with $\tau \neq 0$ be a real-analytic 
	isometric immersion. Then $f$ admits a positive Bonnet mate if and only if $f$ 
	has constant mean curvature and is invariant under a one-parameter group of 
	isometries.
	
	If $f$ has constant principal curvatures, then there exists a unique continuous 
	$1$-parameter family of positive Bonnet mates, parametrized by $\mathbb{S}^{1}$. 
	Otherwise, the positive Bonnet mate is unique.
\end{theorem}

\subsubsection{Uniqueness of the Bonnet problem: negative case}
Given an orientable Riemannian surface $(\Sigma, \mathrm{d} s^2)$, suppose there exist two isometric immersions $f$, $\widetilde{f}: \Sigma \rightarrow \mathbb{E}(\kappa,\tau)$ which are negative Bonnet mates, i.e, $\widetilde{\nu}_3 = \nu_3$ and $\widetilde{J} = - J$. In this case, there exist an angle function $\theta:\Sigma \to \mathbb{R} \text{ (mod }2\pi)$ and a field of isometric endomorphisms $\text{Sym}_\theta$ such that $\text{Sym}_\theta (e_i) = \widetilde{e}_i$ for $i\in \{1,2\}$ (given by~\eqref{eq:frames negative orientation}) and $\text{Sym}_\theta \circ \text{Sym}_\theta = \text{Id}$ (so it is an axial isometry at each point).

In this negative case, the surfaces admitting a Bonnet mate are not necessarily of constant mean curvature so we have $\Sigma = M_1 \cup M_2 \cup M_3$ and we will have three cases.
\begin{lemma}\label{lemma:bonnet mate cases tau neq 0 negative}
	Let $(\Sigma, \mathrm{d} s^2)$ be an orientable Riemannian surface and let $f:\Sigma \rightarrow \mathbb{E}(\kappa,\tau)$ and $\widetilde{f}:\Sigma \rightarrow \mathbb{E}(\kappa,\tau)$ be negative Bonnet mates. Then one of the following holds.
	\begin{itemize}
		\item[$(a)$] If $\Sigma = \overline{M}_1$, we have two possible subcases:
		\begin{itemize}
			\item[$(a_1)$] The relation \textit{(1)} in Lemma~\ref{lemma:first cases} holds, up to ambient isometries, and $\theta$ is uniquely determined by
			\begin{equation}\label{eq:angle negative nabla nu neq 0 T1 T2}
				\begin{aligned}
					\cos\theta &= \frac{- 2 (\kappa_1 + \kappa_2)\tau \langle T_3, e_1\rangle \langle T_3, e_2\rangle + (\kappa_1^2 - \tau^2)\langle T_3, e_1\rangle^2 + (-\kappa_2^2 + \tau^2) \langle T_3, e_2\rangle^2}{2(-\kappa_1 + \kappa_2)\tau \langle T_3, e_1\rangle \langle T_3, e_2\rangle + (\kappa_1^2 + \tau^2)\langle T_3, e_1\rangle^2 + (\kappa_2^2 + \tau^2)\langle T_3, e_2\rangle^2},\\
					\sin\theta &= \frac{- 2(\kappa_2 \langle T_3, e_2\rangle + \tau \langle T_3, e_1\rangle)(-\kappa_1 \langle T_3, e_1\rangle + \tau \langle T_3, e_2\rangle)}{2(-\kappa_1 + \kappa_2)\tau \langle T_3, e_1\rangle \langle T_3, e_2\rangle + (\kappa_1^2 + \tau^2)\langle T_3, e_1\rangle^2 + (\kappa_2^2 + \tau^2)\langle T_3, e_2\rangle^2},
				\end{aligned}
			\end{equation}
			in which case $f$ has at most one Bonnet mate.
			\item[$(a_2)$] The immersion must have constant mean curvature, the relation \textit{(3)} in Lemma~\ref{lemma:first cases} holds, up to ambient isometries, and $\theta$ is uniquely determined by \begin{equation}\label{eq:angle negative nabla nu neq 0 T1 T2 raros}
				\cos\theta = \frac{-2H}{\sqrt{4H^2 + 4\tau^2}}, \quad \sin\theta = \frac{-2\tau}{\sqrt{H^2 + \tau^2}}.
			\end{equation} 
			in which case $f$ has at most one Bonnet mate.
		\end{itemize}
		\item[$(b)$] If $\Sigma = \overline{M}_2$, the relation \textit{(1)} in Lemma~\ref{lemma:first cases} holds, up to ambient isometries, and $\theta$ is uniquely determined by~\eqref{eq:angle theta tau 0 nabla nu 0}. As a result, $f$ has at most one Bonnet mate.
		
		\item[$(c)$] If $\Sigma = \overline{M}_3$, the immersion has constant principal curvatures, in which case, $f$ has at most a one-parameter family of Bonnet mates parameterized by $\mathbb{S}^1$.
	\end{itemize}
\end{lemma}
\begin{proof}
	By matching the components in the identity $\nabla\widetilde{\nu}_3 = \nabla\nu_3$, and
	expressing them in the frame $\{e_1, e_2\}$ by means of~\eqref{eq:frames negative orientation},
	we obtain the following linear system in the unknowns $\cos\theta$ and $\sin\theta$:
	\begin{align*}
		-\kappa_1  \langle T_3, e_1\rangle + \tau \langle T_3, e_2\rangle  &= (-\kappa_1\langle \widetilde{T}_3, \widetilde{e}_1\rangle + \tau \langle \widetilde{T}_3, \widetilde{e}_2\rangle)\cos\theta + (- \kappa_2 \langle \widetilde{T}_3, \widetilde{e}_2\rangle - \tau \langle \widetilde{T}_3, \widetilde{e}_1\rangle)\sin\theta,\\[0.3em]
		- \kappa_2 \langle T_3, e_2\rangle -\tau \langle T_3, e_1\rangle  &= (\kappa_2 \langle \widetilde{T}_3, \widetilde{e}_2\rangle + \tau \langle \widetilde{T}_3, \widetilde{e}_1\rangle) \cos\theta + (- \kappa_1 \langle \widetilde{T}_3, \widetilde{e}_1\rangle + \tau \langle \widetilde{T}_3, \widetilde{e}_2\rangle)\sin\theta.
	\end{align*}
	\noindent For each choice of $\langle \widetilde{T}_3, \widetilde{e}_1\rangle$ and
	$\langle \widetilde{T}_3, \widetilde{e}_2\rangle$, the system has a unique solution for
	$\cos\theta$ and $\sin\theta$ precisely when $||\nabla\nu_3||^2 \neq 0$.  
	We now analyse the different cases in the statement.
	
	\noindent $(a)$ If $\Sigma = \overline{M}_1$, the system admits a unique solution for each of the
	cases in Lemma~\ref{lemma:first cases}.
	
	$(a_1)$ In case \emph{(1)}, where 
	$\langle \widetilde{T}_3, \widetilde{e}_1\rangle = \langle T_3, e_1\rangle$ and 
	$\langle \widetilde{T}_3, \widetilde{e}_2\rangle = \langle T_3, e_2\rangle$,
	the corresponding angle is given explicitly by~\eqref{eq:angle negative nabla nu neq 0 T1 T2}.
	If we are instead in case \emph{(2)}, that is,
	$\langle \widetilde{T}_3, \widetilde{e}_1\rangle = -\langle T_3, e_1\rangle$ and
	$\langle \widetilde{T}_3, \widetilde{e}_2\rangle = -\langle T_3, e_2\rangle$,
	then $\theta$ differs from the previous solution by~$\pi$ and the resulting immersion is
	congruent to the former.  
	In both situations, $\widetilde{T}_3$ is completely determined by $	\widetilde{T}_3 = \langle T_3, e_1\rangle \mathrm{Sym}_\theta e_1
	+ \langle T_3, e_2\rangle \mathrm{Sym}_\theta e_2
	= \mathrm{Sym}_\theta T_3$ and therefore uniqueness holds.
	
	$(a_2)$ In case \emph{(3)}, where
	\[
	\langle \widetilde{T}_3, \widetilde{e}_1\rangle =
	\frac{-H \langle T_3, e_1\rangle + \tau \langle T_3, e_2\rangle}{\sqrt{H^2 + \tau^2}},\qquad
	\langle \widetilde{T}_3, \widetilde{e}_2\rangle =
	\frac{H \langle T_3, e_2\rangle + \tau \langle T_3, e_1\rangle}{\sqrt{H^2 + \tau^2}},
	\]
	the corresponding solutions for $\theta$ are those given in
	\eqref{eq:angle negative nabla nu neq 0 T1 T2 raros}.  
	The fundamental data are then fully determined, and substituting into
	$\widetilde{T}_3 = \langle T_3, e_1\rangle \mathrm{Sym}_\theta e_1
	+ \langle T_3, e_2\rangle \mathrm{Sym}_\theta e_2$ yields $	\widetilde{T}_3 = \cos(2\theta)\, T_3 + \sin(2\theta)\, JT_3
	= \mathrm{Rot}_{2\theta}T_3$. Using identity \textsc{(i)} in Lemma~\ref{th:adapted theorem}, we find that $\theta$ is
	constant, hence $H$ is constant as well.  
	Since both $\widetilde{S}$ and $\widetilde{T}_3$ are obtained from $S$ and $T_3$ by a constant
	rotation, this corresponds to Daniel’s correspondence~\cite[Prop.~5.1]{Dan07}. This proves item $(a_2)$, since case \emph{(4)} differs only by an additive $\pi$ in $\theta$,
	and therefore gives congruent immersions.
	
	\noindent $(b)$ If $\Sigma = \overline{M}_2$, we must have
	$\langle \nabla H, T_3\rangle = \langle \nabla H, \widetilde{T}_3\rangle$, exactly as in
	Lemma~\ref{lemma:bonnet mate cases tau 0}.  
	Since 
	$\langle \widetilde{T}_3, \widetilde{e}_1\rangle = \langle T_3, e_1\rangle$ and
	$\langle \widetilde{T}_3, \widetilde{e}_2\rangle = \langle T_3, e_2\rangle$,
	we again obtain $\widetilde{T}_3 = \mathrm{Sim}_\theta T_3$.  
	The angle equation has the two solutions: $\theta = 0$, giving congruent immersions, and
	the value in~\eqref{eq:angle theta tau 0 nabla nu 0}, which yields at most one Bonnet mate.
	
	\noindent $(c)$ If $\Sigma = \overline{M}_3$, by repeating the argument from the case $\tau=0$
	(see Lemma~\ref{lemma:bonnet mate cases tau 0}), one finds that there is at most a
	one-parameter family of Bonnet mates. 
\end{proof}
With this result, we can now complete the Bonnet problem in the remaining case for $\mathbb{E}(\kappa,\tau)$.
\begin{theorem}\label{teo2}
	Let $(\Sigma, \mathrm{d}s^{2}, J)$ be an oriented Riemannian surface, and let 
	$f:\Sigma \to \mathbb{E}(\kappa,\tau)$ with $\tau \neq 0$ be a real-analytic 
	isometric immersion. Then $f$ admits a negative Bonnet mate if and only if 
	$f$ has constant mean curvature or is invariant under a one-parameter group 
	of isometries.
	
	If $f$ has constant principal curvatures, then there exists exactly an $\mathbb{S}^1$-family of negative Bonnet mates. If $f$ has nonzero constant mean curvature $H \neq 0$ and is invariant under a 
	one-parameter group of isometries, then it has exactly two negative Bonnet mates. 
	In all other cases, the negative Bonnet mate is unique.
\end{theorem}
\begin{proof}
	The proof is analogous to the case $\tau = 0$, although several additional	details must be taken into account.  In this setting, surfaces invariant under a one-parameter group of isometries admit 
	negative Bonnet mates; likewise, constant mean curvature surfaces also possess them, although the transformations involved are quite different 
	(compare~\eqref{eq:angle negative nabla nu neq 0 T1 T2} with \eqref{eq:angle negative nabla nu neq 0 T1 T2 raros}). On the intersection of these two families, each surface admits two negative Bonnet mates, one associated with each transformation. The uniqueness in each case follows from Lemma~\ref{lemma:bonnet mate cases tau neq 0 negative}.
	
	For the case of constant principal curvatures, an $\mathbb{S}^1$-family of Bonnet mates is obtained by first applying twin correspondence (see~\cite{Dan07}), and then applying, on the resulting surface, the same argument as in Theorem~\ref{th:main theorem tau 0}.
\end{proof}
\begin{remark}
	In the case $\tau \neq 0$, these Bonnet mates intersect with the recently discovered class of angular companions introduced in~\cite{CMS}. In Example~3.11, these cylinders are not CMC surfaces, but they admit Bonnet mates. Moreover, when working with the distinguished frame $\{E_1, E_2, E_3\}$ (the one given by Milnor~\cite{Milnor}), the coordinates of the normal vector $N$ with respect to this frame remain unchanged. It should also be emphasized that the class of angular mates and that of Bonnet mates are not contained one in the other. \hfill $\blacksquare$
\end{remark}
In particular we solve the Chern problem in $\mathbb{E}(\kappa,\tau)$ spaces for $\tau \neq 0$.
\begin{corollary}\label{cortauneq0}
	The only immersions admitting an isometric continuous deformation preserving the principal curvatures in $\mathbb{E}(\kappa,\tau)$ with $\tau \neq 0$ (not by ambient isometries) are those with constant principal curvatures.
\end{corollary}
\section{The Bonnet problem in $\mathrm{Sol}_3$}

Given an orientable Riemannian surface $(\Sigma, \mathrm{d} s^2)$, we assume the existence of two isometric immersions 
$f,\widetilde{f}: \Sigma \rightarrow \mathrm{Sol}_3$ which are Bonnet mates. Let $J$ and $\widetilde{J}$ denote the 
orientations on $\Sigma$ induced by $f$ and $\widetilde{f}$, respectively. 
Comparing the Gauss equations for both immersions, which share the Gauss curvature $K$, and using Proposition~\ref{prop:compatibility Sol3}, we have
\[
\det(S) - \mu^2 + 2\mu^2 \nu_3^2 
= K 
= \det(\widetilde{S}) - \mu^2 + 2\mu^2 \widetilde{\nu}_3^2.
\]
\noindent As in the $\mathbb{E}(\kappa,\tau)$ case, we conclude that 
$\widetilde{\nu}_3^2 = \nu_3^2$. Likewise, from the algebraic relations 
\eqref{eq:algebraic relations}, we observe that
\[
\langle \widetilde{T}_3, \widetilde{T}_3\rangle 
= 1 - \widetilde{\nu}_3^2 = 1 - \nu_3^2 
= \langle T_3, T_3\rangle.
\]
Hence, there exists an angle function $\psi : \Sigma \rightarrow \mathbb{R} \pmod{2\pi}$ such that 
$\widetilde{T}_3 = \mathrm{Rot}_\psi\, T_3$.  
By the same algebraic relations, we also deduce 
$\widetilde{\nu}_1^2 + \widetilde{\nu}_2^2 = 1 - \nu_3^2 = \nu_1^2 + \nu_2^2$, and therefore there exists another angle 
function $\phi : \Sigma \rightarrow \mathbb{R} \pmod{2\pi}$ such that
\begin{equation}
	\widetilde{\nu}_1 = \cos\phi\, \nu_1 - \sin\phi\, \nu_2, 
	\qquad 
	\widetilde{\nu}_2 = \sin\phi\, \nu_1 + \cos\phi\, \nu_2.
\end{equation}

Moreover, by Remark~\ref{remark:isometries and fundamental data in Sol3}, 
we assume that $\widetilde{J} = J$ and 
$\widetilde{\nu}_3 = \nu_3$ (note that in $\mathrm{Sol}_3$ there exist 
orientation-reversing isometries). Thus, we fix the structure $(\Sigma, \mathrm{d} s^2, J)$ and assume the Bonnet mates are positive.

We first study Bonnet mates for surfaces with constant left-invariant Gauss map.

\begin{lemma}\label{lemma:nu1 nu2 nu3 in Sol3}
	If an immersion $f$ has constant left-invariant Gauss map, then it admits a one-parameter family of continuous deformations preserving the principal curvatures.
\end{lemma}
\begin{proof}
	In Subsection~\ref{subsec:description Sol3 type spaces} we noted that, by~\cite{CMS}, surfaces with constant Gauss map are precisely those satisfying $X_1 = X_2 = X_3 = 0$, and that such surfaces are extrinsically homogeneous and have constant principal curvatures. As in $\mathbb{E}(\kappa,\tau)$ case, we fix a point $p$ and take an open neighbourhood $U$ invariant under an intrinsic rotation $\phi_{\theta_0}$ fixing $p$. In this setting, the immersions $f: U \to \mathrm{Sol}_3$ and $f\circ \phi_{\theta_0}: U \to \mathrm{Sol}_3$ are not congruent (otherwise, $\mathrm{Sol}_3$ would not have a discrete stabilizer). Since the principal curvatures are constant, there exists an $\mathbb{S}^1$-family of continuous deformations preserving them. Nevertheless, their images in $\mathrm{Sol}_3$ are congruent as subsets.
\end{proof}
From now on, we work on a connected open subset of $\Sigma$ where the left-invariant Gauss map $N$ is not constant. Recall that, in particular, one must have $\nu_1^2, \nu_2^2, \nu_3^2 \neq 1$ on a dense set. We now look for necessary and sufficient conditions for an isometric immersion $f$ to admit a Bonnet mate.
 
The main idea in the calculations below is that, assuming $f$ admits a Bonnet mate $\widetilde{f}$, the compatibility equations of Theorem~\ref{thm:compatibility equations Sol3} for $\widetilde{f}$ can be written solely in terms of the fundamental data of $f$ and the functions $\phi$ and $\psi$.

As a first step, we express $\widetilde{X}_3$ in terms of $X_3$ and $\phi$. 
It is easy to check, using relations~\eqref{eq:relaciones nutildes nu Sol}, that
\begin{equation}\label{eq:relaciones nutildes nu Sol}
	\begin{aligned}
		\widetilde{\nu}_2\nabla \widetilde{\nu}_1 - \widetilde{\nu}_1 \nabla \widetilde{\nu}_2 
		- \nu_2 \nabla \nu_1 + \nu_1 \nabla \nu_2 
		&= -(1-\nu_3^2)\nabla \phi,\\
		\langle \nabla \widetilde{\nu}_2, J\nabla \widetilde{\nu}_1\rangle 
		- \langle \nabla \nu_2, J\nabla \nu_1\rangle 
		&= \nu_3 \langle \nabla \nu_3, J\nabla \phi\rangle.
	\end{aligned}
\end{equation}
From the first identity of \eqref{eq:relaciones nutildes nu Sol}, we immediately obtain
\begin{equation}\label{eq:expresion nabla phi}
	(1-\nu_3^2)\nabla \phi 
	= X_3 - \widetilde{X}_3
	= X_3 + 2H\, \mathrm{Rot}_\psi JT_3 - \zeta\, \mathrm{Rot}_\psi T_3,
\end{equation}
where in the last equality we used the expression for $\widetilde{X}_3$ from 
Theorem~\ref{thm:compatibility equations Sol3}, and recall that 
$\zeta = \mu(\nu_1^2 - \nu_2^2)$.

Since $\widetilde{T}_3 = \mathrm{Rot}_\psi T_3$, the covariant derivative of $\widetilde{T}_3$ can be expressed 
in terms of the covariant derivative of $T_3$ and $\nabla\psi$ as $\nabla_X \widetilde{T}_3 
= \langle \nabla\psi, X\rangle J\widetilde{T}_3 
+ \mathrm{Rot}_\psi \nabla_X T_3$. Multiplying both sides by $J\widetilde{T}_3$ and using the second identity in 
Theorem~\ref{thm:compatibility equations Sol3}, the left-hand side becomes
\[
\langle \nabla_X \widetilde{T}_3, J\widetilde{T}_3\rangle 
= \big\langle\, 
\nu_3 (\widetilde{\nu}_1 \nabla\widetilde{\nu}_2 - \widetilde{\nu}_2 \nabla\widetilde{\nu}_1) 
+ (1-\nu_3^2)\mu(\widetilde{\nu}_2 \widetilde{T}_2 - \widetilde{\nu}_1 \widetilde{T}_1), 
X \big\rangle .
\]
On the right-hand side we obtain
\begin{align*}
	(1-\nu_3^2)\langle \nabla\psi, X\rangle + \langle \nabla_X T_3, JT_3\rangle
	&= (1-\nu_3^2)\langle \nabla\psi, X\rangle + \big\langle X,\, 
	\nu_3(\nu_1\nabla\nu_2 - \nu_2\nabla\nu_1)
	+(1-\nu_3^2)\mu(\nu_2 T_2 - \nu_1 T_1)
	\big\rangle .
\end{align*}

Hence, assuming $\nu_3^2 \neq 1$, the last identity is equivalent to
\begin{equation}\label{eq:expresion para suma nu3 nabla phi nabla theta}
	\nabla \psi - \nu_3 \nabla \phi
	= \mu(\widetilde{\nu}_2\widetilde{T}_2 - \widetilde{\nu}_1\widetilde{T}_1)
	- \mu(\nu_2 T_2 - \nu_1 T_1).
\end{equation}
\noindent One verifies, using the relations in \eqref{eq:T1 T2 en T3}, that
\[
\nu_3 X_3 + (1-\nu_3^2)\mu(\nu_1 T_1 - \nu_2 T_2)
= -2H\nu_3 JT_3 + 2\mu\nu_1\nu_2 JT_3
= -\mathrm{div}(T_3)\, JT_3,
\]
and an analogous expression holds for the tilded quantities. 
Using this and \eqref{eq:expresion para suma nu3 nabla phi nabla theta}, we obtain
\begin{equation}\label{eq:expresion nabla psi}
	\begin{aligned}
		(1-\nu_3^2)\nabla \psi
		&= \mathrm{div}(\widetilde{T}_3)\, J\widetilde{T}_3 
		- \mathrm{div}(T_3)\, JT_3 = \mathrm{div}(\mathrm{Rot}_\psi T_3)\, \mathrm{Rot}_\psi JT_3 
		- \mathrm{div}(T_3)\, JT_3.
	\end{aligned}
\end{equation}
Note that equations \eqref{eq:expresion nabla phi} and \eqref{eq:expresion nabla psi}, whose unknowns are 
$\phi$ and $\psi$, depend only on $f$. We therefore deduce the following result.
\begin{theorem}\label{pr:characterization Sol Bonnet mates}
	Let $(\Sigma, \mathrm{d} s^2, J)$ be an oriented Riemannian surface and 
	$f:\Sigma \rightarrow \mathrm{Sol}_3$ an isometric immersion whose left-invariant Gauss map $N$ 
	is not constant on any open set. Then $f$ admits a Bonnet mate 
	$\widetilde{f}:\Sigma \rightarrow \mathrm{Sol}_3$ if and only if 
	there exist nontrivial functions $\phi$ and $\psi$ defined on $\Sigma$ satisfying 
	equations~\eqref{eq:expresion nabla phi} and \eqref{eq:expresion nabla psi}.
\end{theorem}

This theorem characterizes Bonnet mates in $\mathrm{Sol}_3$ in analogy to Chern in $\mathbb{R}^3$. We now carry out some computations that will be useful in establishing an upper bound for 
the number of Bonnet mates of an isometric immersion into $\mathrm{Sol}_3$. The objective is to rewrite all relevant objects and equations from 
Theorem~\ref{thm:compatibility equations Sol3} in terms of $T_3$, $JT_3$, and the angle functions 
$\phi$ and $\psi$, so that all tilded data depend only on these two functions.

\begin{lemma}\label{lemma:expresiones nabla nu1, nabla nu 2}
	Let $(\Sigma, \mathrm{d} s^2, J)$ be an oriented Riemannian surface and let $f$ be an isometric immersion 
	into $\mathrm{Sol}_3$ with $\nu_3^2 \neq 1$. Then the following identities hold:
	\begin{align*}
		&\langle \nabla\nu_1, T_3\rangle = \langle \nabla\nu_3, T_1\rangle + \zeta \nu_2,\\
		&\langle \nabla\nu_2, T_3\rangle = \langle \nabla\nu_3, T_2\rangle - \zeta \nu_1,\\
		&\langle \nabla\nu_1, JT_3\rangle = \langle \nabla\nu_3, JT_1\rangle - 2H \nu_2,\\
		&\langle \nabla\nu_2, JT_3\rangle = \langle \nabla\nu_3, JT_2\rangle + 2H \nu_1.
	\end{align*}
	As an immediate consequence, using equations~\eqref{eq:T1 T2 en T3}, one obtains
	\begin{align*}
		&\langle \nu_1\nabla\nu_1 - \nu_2\nabla\nu_2, JT_3\rangle = -\Bigg\langle 
		\frac{\nabla\nu_3}{1-\nu_3^2},\,
		(\nu_1^2 - \nu_2^2)\nu_3 JT_3 + 2\nu_1\nu_2 T_3 
		\Bigg\rangle 
		- 4H\nu_1\nu_2,\\
		&\langle \nu_1\nabla\nu_2 + \nu_2\nabla\nu_1, T_3\rangle = -\Bigg\langle 
		\frac{\nabla\nu_3}{1-\nu_3^2},\,
		2\nu_1\nu_2\nu_3 T_3 + (\nu_1^2-\nu_2^2) JT_3 
		\Bigg\rangle 
		+ 4\mu \nu_1^2 \nu_2^2 - \mu(1-\nu_3^2)^2.
	\end{align*}
\end{lemma}
\begin{proof}
	In all cases we apply the equations \textsc{(iii)} of 
	Proposition~\ref{prop:compatibility Sol3}.  
	For instance, for the first identity we compute
	\begin{align*}
		\langle \nabla\nu_1, T_3\rangle 
		&= \langle -ST_1 - \mu \nu_3 T_2, T_3\rangle 
		= -\langle ST_1, T_3\rangle - \mu \nu_3 \langle T_2, T_3\rangle\\
		&= -\langle T_1, ST_3\rangle + \mu \nu_2 \nu_3^2\\
		&= -\langle T_1, -\nabla\nu_3 + \mu \nu_1 T_2 + \mu \nu_2 T_1\rangle + \mu \nu_2 \nu_3^2
	\end{align*}
	and the equation holds applying the algebraic relations~\eqref{eq:algebraic relations}. The second identity in the statement is proved in the same way, and the remaining ones follow similarly, using also the relation $JS + SJ = 2HJ$.
\end{proof}

\begin{lemma}\label{th:teorema principal}
	Let $(\Sigma, \mathrm{d} s^2, J)$ be an oriented Riemannian surface and let 
	$f:\Sigma \rightarrow \mathrm{Sol}_3$ be an isometric immersion whose left-invariant Gauss map is 
	not constant on any open set. If $f$ admits a Bonnet mate $\widetilde{f}$, then the following equalities hold:
	\begin{align*}
		&\frac{1}{1-\nu_3^2}\langle \nabla\nu_3, \mathrm{div}(T_3)T_3 \rangle 
		+ \langle \nabla H, T_3\rangle - 6\mu H\nu_1\nu_2\\
		&\qquad =
		\frac{1}{1-\nu_3^2}\langle \nabla\nu_3, \mathrm{div}(\widetilde{T}_3)\widetilde{T}_3 \rangle 
		+ \langle \nabla H, \widetilde{T}_3\rangle 
		- 6\mu H\widetilde{\nu}_1 \widetilde{\nu}_2,
	\end{align*}
	\begin{align*}
		&\frac{1}{1-\nu_3^2}\langle \nabla\nu_3,\, H(1+\nu_3^2)T_3 + \zeta JT_3\rangle 
		+ \nu_3\langle \nabla H, T_3\rangle 
		- 2\mu^2\nu_1^2\nu_2^2 - 4\mu H \nu_1\nu_2\nu_3\\
		&\qquad =
		\frac{1}{1-\nu_3^2}\langle \nabla\nu_3,\, H(1+\nu_3^2)\widetilde{T}_3 + \widetilde{\zeta}J\widetilde{T}_3\rangle 
		+ \nu_3\langle \nabla H, \widetilde{T}_3\rangle 
		- 2\mu^2\widetilde{\nu}_1^2 \widetilde{\nu}_2^2 
		- 4\mu H \widetilde{\nu}_1\widetilde{\nu}_2\nu_3.
	\end{align*}
\end{lemma}

\begin{remark}
	Although the lemma states that the functions $\phi$ and $\psi$ satisfy these equations, 
	their dependence does not appear explicitly. For clarity, we omit it here, but it should be 
	understood that since $\widetilde{T}_3 = \mathrm{Rot}_\psi T_3$ and 
	$\widetilde{\nu}_1$, $\widetilde{\nu}_2$ are expressed in terms of $\nu_1$, $\nu_2$, and $\phi$ 
	(see equation~\eqref{eq:relaciones nutildes nu Sol}), the equations do indeed depend on 
	$\phi$ and $\psi$. \hfill $\blacksquare$
\end{remark}
\begin{proof}
	Given two isometric immersions $f, \widetilde{f}:\Sigma \rightarrow \mathrm{Sol}_3$ that 
	form a Bonnet pair, Theorem~\ref{pr:characterization Sol Bonnet mates} ensures the 
	existence of angle functions satisfying~\eqref{eq:expresion nabla phi} and 
	\eqref{eq:expresion nabla psi}.  
	These functions must satisfy $\operatorname{div}(J \nabla \phi)=0$ and $\operatorname{div}(J \nabla \psi)=0$. 
	  
	For $\phi$ we obtain
	\begin{align*}
		0 &= \operatorname{div}(J\nabla \phi)
		= \Big\langle \nabla \tfrac{1}{1-\nu_3^2},\, JX_3 - J\widetilde{X}_3 \Big\rangle
		+ \tfrac{1}{1-\nu_3^2}\operatorname{div}(JX_3 - J\widetilde{X}_3)\\
		&= 2\frac{\nu_3}{1-\nu_3^2}\langle \nabla\nu_3, J\nabla\phi\rangle
		+ \frac{1}{1-\nu_3^2}\operatorname{div}(JX_3 - J\widetilde{X}_3).
	\end{align*}
	Multiplying by $1-\nu_3^2$ and using item (a) in Theorem~\ref{thm:compatibility equations Sol3}, we have
	\begin{align*}
		0 = 2\nu_3 \langle \nabla\nu_3, J\nabla\phi\rangle
		+ \operatorname{div}(JX_3 - J\widetilde{X}_3) &= 2\nu_3 \langle \nabla\nu_3, J\nabla\phi\rangle
		+ 2\langle \nabla H, T_3\rangle + 2H\operatorname{div}(T_3) + \langle \nabla\zeta, JT_3\rangle\\
		&\qquad
		- 2\langle \nabla H, \widetilde{T}_3\rangle - 2H\operatorname{div}(\widetilde{T}_3)
		- \langle \nabla\widetilde{\zeta}, J\widetilde{T}_3\rangle .
	\end{align*}
	Grouping the terms with and without tilde, we obtain
	\begin{align*}
		&\frac{2\nu_3}{1-\nu_3^2}\langle \nabla\nu_3, J\widetilde{X}_3\rangle
		+ 2\langle \nabla H, \widetilde{T}_3\rangle
		+ 2H\operatorname{div}(\widetilde{T}_3)
		+ \langle \nabla\widetilde{\zeta}, J\widetilde{T}_3\rangle\\
		&= \frac{2\nu_3}{1-\nu_3^2}\langle \nabla\nu_3, JX_3\rangle
		+ 2\langle \nabla H, T_3\rangle
		+ 2H\operatorname{div}(T_3)
		+ \langle \nabla\zeta, JT_3\rangle.
	\end{align*}
	Using Lemma~\ref{lemma:expresiones nabla nu1, nabla nu 2} to rewrite  
	$\langle \nabla(\nu_1^2-\nu_2^2), JT_3\rangle$  
	and   
	$\langle \nabla(\widetilde{\nu}_1^2-\widetilde{\nu}_2^2), J\widetilde{T}_3\rangle$,  
	together with the expressions of $X_3$ and $\widetilde{X}_3$, yields the first identity.
	
	The computation for $\psi$ proceeds in the same manner:
	\begin{align*}
		0 
		&= \operatorname{div}(J\nabla\psi)
		= \operatorname{div}\!\left(
		\tfrac{1}{1-\nu_3^2} \big(-\operatorname{div}(\widetilde{T}_3)\widetilde{T}_3 
		+ \operatorname{div}(T_3)T_3\big)
		\right)\\
		&= \tfrac{1}{1-\nu_3^2}\big(\langle 2\nu_3\nabla\nu_3, J\nabla\psi\rangle
		- \langle \nabla\operatorname{div}(\widetilde{T}_3),\widetilde{T}_3\rangle
		- \operatorname{div}(\widetilde{T}_3)^2 + \langle \nabla\operatorname{div}(T_3),T_3\rangle
		+ \operatorname{div}(T_3)^2
		\big).
	\end{align*}
	Multiplying by $1-\nu_3^2$ and grouping tilded and non-tilded terms, we get
	\begin{align*}
		&\frac{2\nu_3}{1-\nu_3^2}\langle \nabla\nu_3,\, \operatorname{div}(\widetilde{T}_3)\widetilde{T}_3\rangle
		+ \langle \nabla\operatorname{div}(\widetilde{T}_3),\widetilde{T}_3\rangle
		+ \operatorname{div}(\widetilde{T}_3)^2\\
		&=\frac{2\nu_3}{1-\nu_3^2}\langle \nabla\nu_3,\, \operatorname{div}(T_3)T_3\rangle
		+ \langle \nabla\operatorname{div}(T_3),T_3\rangle
		+ \operatorname{div}(T_3)^2.
	\end{align*}
	
	Finally, if we expand $\nabla\operatorname{div}(T_3)$ using identity~\eqref{eq:divergences T JT in Sol} and apply Lemma~\ref{lemma:expresiones nabla nu1, nabla nu 2}, we obtain the second identity.
\end{proof}

Given an immersion $f:\Sigma \rightarrow \mathrm{Sol}_3$, we introduce the notation 
$A_1 = \langle \nabla\nu_3, T_3\rangle$, 
$A_2 = \langle \nabla\nu_3, JT_3\rangle$, 
$B_1 = \langle \nabla H, T_3\rangle$, 
$B_2 = \langle \nabla H, JT_3\rangle$.  
We also define the auxiliary functions
\begin{align*}
	f_1 = \frac{1}{1-\nu_3^2}\langle \nabla\nu_3, \mathrm{div}(T_3)T_3\rangle 
	+ \langle \nabla H, T_3\rangle - 6\mu H\nu_1\nu_2 = \frac{\mathrm{div}(T_3)}{1-\nu_3^2}A_1 + B_1 - 6H\nu_1\nu_2,
\end{align*}\vspace{-0.37cm}
\begin{align*}
	f_2 
	&=\frac{1}{1-\nu_3^2}
	\langle \nabla\nu_3,\, H(1+\nu_3^2)T_3 + \zeta JT_3\rangle 
	+ \nu_3\langle \nabla H, T_3\rangle 
	- 2\mu^2 \nu_1^2\nu_2^2 - 4\mu H\nu_1\nu_2\nu_3\\
	&= \frac{1}{1-\nu_3^2}\big(H(1+\nu_3^2)A_1 + \zeta A_2\big) + \nu_3 B_1 
	- 2\mu^2\nu_1^2\nu_2^2 - 4\mu H\nu_1\nu_2\nu_3.
\end{align*}
\noindent Thus, the two identities in Lemma~\ref{th:teorema principal} may be written simply as 
$f_i = \widetilde{f}_i$ for $i\in\{1,2\}$. With this, we can derive a result that provides an upper bound for the number of Bonnet mates.

\begin{theorem}\label{th:last theorem}
	Consider an oriented Riemannian surface $(\Sigma, \mathrm{d} s^2, J)$ and an isometric immersion 
	$f:\Sigma \rightarrow \mathrm{Sol}_3$ whose left-invariant Gauss map is not constant on any open set. 
	Then $f$ admits at most $7$ Bonnet mates.
\end{theorem}
\begin{proof}
	The idea is to reduce the equations in Lemma~\ref{th:teorema principal} to a finite 
	algebraic system and then apply Bézout’s Theorem to estimate the number of possible 
	Bonnet mates.
	
	Let $f:\Sigma \rightarrow \mathrm{Sol}_3$ be an isometric immersion admitting a Bonnet 
	mate $\widetilde{f}$. Then there exist angle functions $\phi$ and $\psi$ satisfying 
	$f_1=\widetilde{f}_1$ and $f_2=\widetilde{f}_2$.  
	To reformulate the problem algebraically, we introduce the variables $	x=\cos\psi$, $y=\sin\psi$, which satisfy the constraint $x^2+y^2=1$, and use that $\nu_1^2+\nu_2^2=\widetilde{\nu}_1^2+\widetilde{\nu}_2^2$. By the relations $B_1 = x\,\widetilde{B}_1 - y\,\widetilde{B}_2$ and $B_2 = y\,\widetilde{B}_1 + x\,\widetilde{B}_2$, we can express $f_1$ and $f_2$ as polynomials in $x$ and $y$. (Although $\widetilde{B}_1$ and $\widetilde{B}_2$ are a priori unknown, both immersions are assumed to have prescribed data; the only unknowns are the angle functions $\phi$ and $\psi$.) Similarly, $\widetilde{f}_1$ and $\widetilde{f}_2$ become polynomials in 	$\widetilde{\nu}_1$ and $\widetilde{\nu}_2$.
	
	In this way we obtain a system of four polynomial equations:
	\begin{align*}
		&P_1 (x,y, \widetilde{\nu}_1, \widetilde{\nu}_2) = 
		f_1(x,y) - \widetilde{f}_1(\widetilde{\nu}_1, \widetilde{\nu}_2)=0, 
		&& P_2 (x,y, \widetilde{\nu}_1, \widetilde{\nu}_2) =
		f_2(x,y) - \widetilde{f}_2(\widetilde{\nu}_1, \widetilde{\nu}_2)=0, \\
		&P_3(x,y)= x^2+y^2 - 1=0,
		&&P_4(\widetilde{\nu}_1, \widetilde{\nu}_2)= 
		\widetilde{\nu}_1^2 + \widetilde{\nu}_2^2 - (\nu_1^2+\nu_2^2)=0.
	\end{align*}
	
	The degrees of $P_1$, $P_2$, $P_3$, and $P_4$ are $2$, $4$, $2$, and $2$, respectively.  
	By Bézout’s Theorem, the system has at most $2\cdot 4 \cdot 2 \cdot 2 = 32$ solutions.  
	However, geometric restrictions imply that not all of these correspond to non-congruent 
	Bonnet mates.  
	Indeed, we are working inside a subgroup of the isotropy group, since we only consider 
	deformations preserving the vertical component of the normal and the orientation, i.e., 
	we assume $\widetilde{\nu}_3 = \nu_3$ and $\widetilde{J}=J$.  
	Using the isotropy group of $\mathrm{Sol}_3$, there are $4$ transformations preserving 
	these data (namely $\mathrm{Id}$, $\Psi_1\circ \Psi_1$, $\Psi_1$, and 
	$\Psi_2\circ \Psi_1 \circ \Psi_2$, up to changing the sign of $N$ in the last two, see Equation~\eqref{eq:isotropy Sol3}).	Hence each Bonnet mate corresponds to congruent solutions of the algebraic system, and 
	the number of distinct Bonnet mates is at most $8$.
	
	Finally, since the trivial solution $(\phi,\psi)=(0,0)$ corresponds to the original immersion $f$, we conclude that there are at most 7 non-trivial Bonnet mates.
\end{proof}
By Lemma~\ref{lemma:nu1 nu2 nu3 in Sol3} and Theorem~\ref{th:last theorem}, we deduce Chern's problem for $\mathrm{Sol}_3$ geometries.
\begin{corollary}\label{corChernSol3}
	The only immersions in $\mathrm{Sol}_3$ that admit a continuous isometric deformation preserving the principal curvatures (not arising from ambient isometries) are exactly those whose left-invariant Gauss map is constant.
\end{corollary}
\begin{remark}
Since isoparametric surfaces (see~\cite{DM}) constitute the solutions to the Chern problem in $\mathbb{E}(\kappa,\tau)$ (including minimal surfaces when $\tau=0$), the natural candidates for isoparametric surfaces in $\mathrm{Sol}_3$ are those with constant left-invariant Gauss map. Whether these are the only such surfaces remains an open question. \hfill $\blacksquare$
\end{remark}

The upper bound on the number of Bonnet mates can be improved, but it cannot be $0$ in general, 
since properly invariant immersions under a one-parameter group of isometries do admit 
Bonnet mates in $\mathrm{Sol}_3$.

\begin{proposition}\label{propSol3inv}
	Let $(\Sigma, \mathrm{d} s^2, J)$ be an oriented real-analytic Riemannian surface and let 
	$f:\Sigma \to \mathrm{Sol}_3$ be an isometric immersion properly invariant under a one-parameter group 
	of isometries. Then $f$ admits at least one Bonnet mate which is properly invariant under a one-parameter group.
\end{proposition}
\begin{proof}
	Let $Z$ be the Killing field induced on $(\Sigma, \mathrm{d}s^{2})$. Recall that, according to Subsection~\ref{subsec:invariantes}, we can define two types of local 
	isometries on $\Sigma$: the \emph{translations} $\varphi_t$ and the 
	\emph{reflections} $\phi_o$. We define the vector field $w = -\tfrac{JZ}{\|Z\|}$ and the field of operators 
	$\mathrm{R} : \mathfrak{X}(\Sigma) \to \mathfrak{X}(\Sigma)$ given by 
	$\mathrm{R}(Y) = \langle Y, w \rangle w - \langle Y, Jw \rangle Jw$ for every 
	$Y \in \mathfrak{X}(\Sigma)$.  
	Geometrically, $\mathrm{R}$ satisfies $\mathrm{R}^2 = \mathrm{Id}$ so it is an \emph{axial reflection} in each tangent space 
	$T_p\Sigma$. It is straightforward to check that this operator is induced by $\phi_o$ so $\mathrm{R}(Y_{\phi_o(p)}) = \mathrm{d}\phi_o(Y_p)$ and $\mathrm{R}\circ S\circ \mathrm{R} = \mathrm{d}\phi_o\circ S\circ \mathrm{d} \phi_o$. Moreover, $\operatorname{trace}(\mathrm{R}\circ S\circ \mathrm{R}) = \operatorname{trace}(S)$ and $\det(\mathrm{R}\circ S\circ \mathrm{R}) = \det(S)$. With this, we also have that
	\begin{equation}\label{last}
		\nabla_{\mathrm{R}X}\mathrm{R}Y
		= \nabla_{\mathrm{d}\phi_o(X)} \mathrm{d} \phi_o(Y)
		= \mathrm{d} \phi_o(\nabla_X Y)
		= \mathrm{R}(\nabla_X Y).
	\end{equation}
	In this proof we apply~\cite[Thm. 3.6]{CMS} and work with the original fundamental data $(J,S,T_\alpha,\nu_\alpha)$ for $\alpha\in\{1,2,3\}$, since in this 
	context it is more convenient than using Theorem~\ref{th:adapted theorem}.  
	We will show that the new set of data
	\[
	(\widetilde{J}, \widetilde{S}, \widetilde{T}_\alpha, \widetilde{\nu}_\alpha)
	= (-J,\, \mathrm{R}\circ S\circ \mathrm{R},\, \mathrm{R}T_\alpha,\, \nu_\alpha)
	\]
	for $\alpha\in\{1,2,3\}$ satisfies the compatibility equations~(iii) of 
	Proposition~\ref{prop:compatibility Sol3} so it defines an isometric immersion and we prove that it shares the principal curvatures.
	
	It is straightforward that the self-adjoint endomorphism $\mathrm{R}\circ S\circ \mathrm{R}$ has the same principal curvatures as $S$. Moreover, the algebraic relations are preserved since $\mathrm{R}$ is an isometry on each 
	tangent space $T_p\Sigma$; in particular 
	$\langle \mathrm{R}T_\alpha,\,\mathrm{R}T_\beta\rangle = \langle T_\alpha,\, T_\beta\rangle$ 
	for all $\alpha,\beta\in\{1,2,3\}$.  
	As these algebraic relations hold, we deduce $\widetilde{J}=-J$; indeed, 
	since $J\mathrm{R} = -\mathrm{R}J$, this choice is consistent, and we have
	\[
	\widetilde{\nu}_3 = \langle \widetilde{J}\widetilde{T}_1, \widetilde{T}_2\rangle
	= \langle -J\mathrm{R}T_1,\, \mathrm{R}T_2\rangle
	= \langle \mathrm{R}JT_1,\, \mathrm{R}T_2\rangle
	= \langle JT_1,\,T_2\rangle
	= \nu_3.
	\]
	
	We now verify the compatibility equations~(iii) in Proposition~\ref{prop:compatibility Sol3}. By~\eqref{last}, we have $\nabla_X (\mathrm{R}T_3)= \mathrm{R}\big(\nabla_{\mathrm{R}X} T_3\big)$. Since $T_3$ satisfies the original compatibility equations, we compute:
	\begin{align*}
		\nabla_X \widetilde{T}_3
		&= \mathrm{R}\big( \nu_3 S(\mathrm{R}X)
		+ \mu \langle \mathrm{R}X, T_2 \rangle T_1
		+ \mu \langle \mathrm{R}X, T_1 \rangle T_2 \big) \\
		&= \nu_3\, \mathrm{R}\big(S(\mathrm{R}X)\big)
		+ \mu \langle X, \mathrm{R}T_2 \rangle\, \mathrm{R}T_1
		+ \mu \langle X, \mathrm{R}T_1 \rangle\, \mathrm{R}T_2 \\
		&= \widetilde{\nu}_3 \widetilde{S}X
		+ \mu \langle X, \widetilde{T}_2\rangle \widetilde{T}_1
		+ \mu \langle X, \widetilde{T}_1\rangle \widetilde{T}_2.
	\end{align*}
	Applying the same reasoning to the remaining equations in (iii), we conclude 
	that~\cite[Thm. 3.6]{CMS} is satisfied for the new fundamental data. As the endomorphism field $\mathrm{R}$ is induced locally by the reflection $\phi_o$, we deduce that the Bonnet mate is also properly invariant since it is locally a reparametrization of a properly invariant immersion.
\end{proof}

\end{document}